\documentclass[reqno,10pt]{amsart}
\usepackage[abs]{overpic} \usepackage{pgf} \usepackage{amsmath,amsthm} \usepackage[active]{srcltx} \usepackage{color,xcolor} \usepackage{epsfig}
\usepackage{graphicx}
\usepackage{amssymb}
\usepackage{latexsym}
\usepackage{pdfpages}
\usepackage{enumerate}
\usepackage{url}
\usepackage{appendix}
\usepackage{enumitem}

\usepackage{ulem}

\usepackage{ifpdf}

\ifpdf 
  \usepackage[hidelinks]{hyperref}
\else 
\fi 

\newcommand{\EEE}{\color{black}}

 \usepackage[usenames,dvipsnames]{pstricks}
 \usepackage{epsfig}
 \usepackage{pst-grad} 
 \usepackage{pst-plot} 

\usepackage{color}

\newtheorem{theorem}{Theorem}[section]
\newtheorem{lemma}[theorem]{Lemma}

\newtheorem{proposition}[theorem]{Proposition}

\newtheorem{remark}[theorem]{Remark}

\newtheorem*{theorem*}{\it Theorem}

\numberwithin{equation}{section}

\def\1{\raisebox{2pt}{\rm{$\chi$}}}

\def\R{\mathbb{R}}

\def\XXint#1#2#3{{\setbox0=\hbox{$#1{#2#3}{\int}$}
     \vcenter{\hbox{$#2#3$}}\kern-.5\wd0}}

\newcommand{\twopartdef}[4]
{
\left\{
		\begin{array}{ll}
			#1 & #2 \\
			#3 & #4
		\end{array}
	\right.
}

\definecolor{violet(ryb)}{rgb}{0.53, 0.0, 0.69}
\usepackage{mathtools}
\mathtoolsset{showonlyrefs}

\usepackage{geometry}
\begin{document}

\title[Double-bubble minimizers]{\bf A characterization of
  $\ell_1$ double bubbles\\  with general interface interaction}

\author[M. Friedrich]{Manuel Friedrich} 
\address[Manuel Friedrich]{Department of Mathematics, Friedrich-Alexander Universit\"at Erlangen-N\"urnberg. Cauerstr.~11,
D-91058 Erlangen, Germany, \& Mathematics M\"{u}nster,  
University of M\"{u}nster, Einsteinstr.~62, D-48149 M\"{u}nster, Germany}
\email{manuel.friedrich@fau.de}

\author[W. G\'orny]{Wojciech G\'orny}
\address[Wojciech G\'orny]{Faculty of Mathematics, Informatics and
  Mechanics, University of Warsaw, Banacha 2, 02-097 Warsaw, Poland
  and Faculty of Mathematics, University of
  Vienna, Oskar-Morgenstern-Platz 1, A-1090 Vienna, Austria}
\email{wojciech.gorny@univie.ac.at}
\urladdr{\url{https://www.mat.univie.ac.at/~wgorny}}

\author[U. Stefanelli]{Ulisse Stefanelli}
\address[Ulisse Stefanelli]{Faculty of Mathematics, University of
  Vienna, Oskar-Morgenstern-Platz 1, A-1090 Vienna, Austria,
Vienna Research Platform on Accelerating
  Photoreaction Discovery, University of Vienna, W\"ahringerstra\ss e 17, 1090 Vienna, Austria,
 \& Istituto di
  Matematica Applicata e Tecnologie Informatiche {\it E. Magenes}, via
  Ferrata 1, I-27100 Pavia, Italy
}
\email{ulisse.stefanelli@univie.ac.at}
\urladdr{\url{http://www.mat.univie.ac.at/~stefanelli}}


\setcounter{tocdepth}{1}


\begin{abstract}
We investigate the optimal arrangements of two planar sets of given volume which are minimizing the $\ell_1$ double-bubble interaction functional. The latter features a competition between the minimization of the $\ell_1$ perimeters of the two sets and the maximization  of their $\ell_1$ interface. We investigate the problem in its full generality for sets of finite perimeter, by considering the whole range of possible interaction intensities and all relative volumes of the two sets. The main result is the complete classification of minimizers.
\end{abstract}

\maketitle
\thispagestyle{empty}


\section{Introduction}

This note is concerned with the planar double-bubble problem with
respect to the $\ell_1$ norm $\| (x_1,x_2)\|_{\ell_1} = |x_1|+|x_2|$. To each configuration
$(A,B)$ consisting of two planar finite perimeter sets
\cite{Ambrosio-Fusco-Pallara}, we associate the energy
\begin{equation}\label{eq:0energy}
E(A,B) :=    \ell_1(\partial^* A  )+\ell_1(\partial^* B  ) + (\eta-2) \, \ell_1(\partial^* A \cap \partial^*B).
\end{equation}
Here, $\partial^*M $ stands for the {\it reduced boundary} of the
finite perimeter set $M\subset {\mathbb R}^2$, and letting
$\nu=(\nu_1,\nu_2)$ be the corresponding (measure-theoretic) {\it
  outer unit normal}, for all rectifiable  $F \subset \partial^* M$  we
indicate 
$$\ell_1(F) = \int_F  (|\nu_1|  + |\nu_2|) \, {\rm d}\mathcal{H}^1$$
the length with respect to the $\ell_1$ norm.

Given the {\it volumes} $V_A,\, V_B >0$ and an {\it interaction
  intensity} $\eta\in (0,2)$, the {\it double-bubble problem} with
respect to the $\ell_1$ norm reads
\begin{equation}
  \label{eq:0double}
  \min\Big\{ E(A,B) \, : \, \text{$A,\, B\subset {\mathbb R}^2$ of
    finite perimeter},  \ A \cap B=\emptyset,  \  {\mathcal L}^2
  (A) \EEE =V_A, \ \text{and} \   {\mathcal L}^2
  (B) \EEE=V_B\Big\}
\end{equation}
where  ${\mathcal L}^2$ \EEE stands for the Lebesgue measure in the plane.

The double-bubble problem can be seen as a schematic model for the
interaction of two systems occupying distinct planar regions. The
minimization of the energy $E$ features the competition between the
first two terms  in \eqref{eq:0energy}, which would favour the configuration consisting of
two separate squares, and the {\it interaction} $\eta$-term, which on the
contrary favours pairs of sets sharing a large portion of the
boundary as $\eta -2<0$.

Our main result Theorem \ref{th:main} provides a complete
characterization of the solutions to the double-bubble problem
\eqref{eq:0double} with respect to all interaction intensities
$\eta\in (0,2)$ and all volume ratios $r=V_B/V_A\in (0,1]$ (we assume
with no loss of generality that $V_A\geq V_B$).  In particular,
\EEE the unique minimizers for given $(r,\eta)$  are classified
\EEE among the three
types of Figure \ref{types}.
\begin{figure}[h]
  \centering
  \pgfdeclareimage[width=100mm]{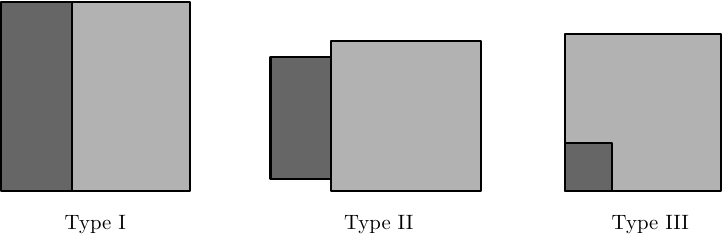}{types}
  \pgfuseimage{types}
  \label{types}
  \caption{Types of minimizers for the double-bubble
     problem \eqref{eq:0double}.}
\end{figure}
More precisely, we identify the regions of the $(r,\eta)$-parameter \EEE space, where
each of the type is favored. The result is depicted in
Figure~\ref{diagram}.  All \EEE phase-separation curves distinguishing types in Figure~\ref{diagram} are explicitly determined  in Section \ref{sec:main} below.

\begin{figure}[h]
  \centering 
  \begin{overpic}[trim=0cm 75mm 0cm 7cm, width=100mm]{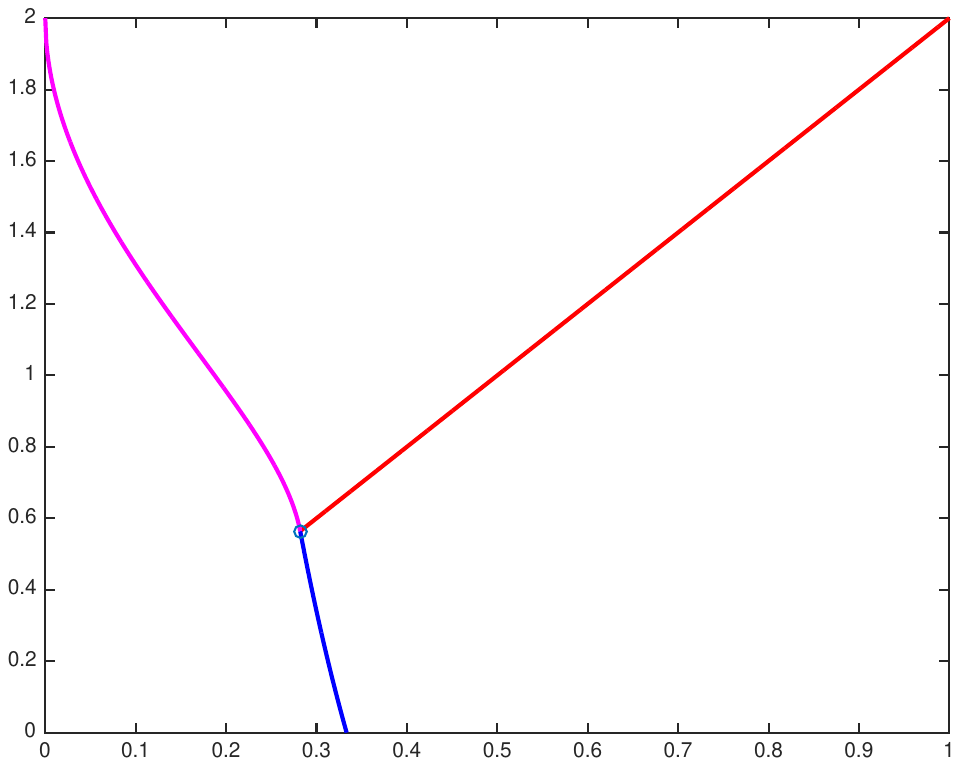}
\put(100,115){Type II}
\put(180,70){Type I}
\put(55,40){Type III}
\put(145,-5){$r$}
\put(30,87){$\eta$}
\end{overpic}

\vspace{0.4cm}

    \caption{Types of minimizers from Figure
      \ref{types} in dependence of $(r,\eta)\in (0,1)\times(0,2)$.}
    \label{diagram} 
  \end{figure}
  
  \EEE

Note that we consider here the whole range of relevant interaction
intensities $\eta \in (0,2)$. In fact, the limiting case $\eta=2$ 
corresponds to no interaction, namely, $E(A,B)=
\ell_1(\partial^*A)+\ell_2(\partial^*B)$,  and the minimizing configuration is two
separate squares of volumes $V_A$ and $V_B$. On the other hand, in the
limiting case $\eta=0$  the minimizing
configuration is such that $A \cup B$ is a square of volume $V_A+
V_B$.

Double-bubble problems \EEE have a long tradition. In the Euclidean setting,
minimisers are known to be enclosed by three spherical
caps, intersecting at an angle of $2\pi/3$. This has been checked in
the plane in \cite{Foisy}, in three-dimensional space in
\cite{Hutchings}, and finally for all dimensions in
\cite{Reichardt}. Double-bubble problems
have been considered in non-Euclidean settings, as well. These include
spherical and hyperbolic spaces 
\cite{Corneli,Corneli2,Cotton,Masters},  hyperbolic surfaces
\cite{Boyer} and cones
 \cite{Lopez,Morgan98}, \EEE the three-dimensional torus \cite{Carrion,Corneli0},  the
Gau\ss\ space \cite{Corneli2,Milman}, and in the anisotropic
Grushin plane \cite{Franceschi}.

Planar double-bubble problems in the case of the $\ell_1$ norm  have already
been considered in the classical case $\eta=1$. The emergence of the
three different minimizers, depending on the relative volume ratio,
has been discussed  by \EEE {\sc Morgan, French, \& Greenleaf}  
\cite{Morgan}. \EEE A new proof of these results, based on different tools,
has been recently presented by {\sc Duncan,
  O'Dwyer, \& Procaccia} \cite{Duncan0}.

Compared with these contributions, we are able to generalize the result in two directions. At first, we are able to weaken the requirements on the class of admissible configurations, which are here solely asked to be of finite perimeter. Note that, in order to make sense to
the double-bubble problem \eqref{eq:0double},  this is a minimal requirement. \EEE Secondly, we can treat here the whole range of interaction intensities $\eta \in (0,2)$, as opposed to  \cite{Duncan0,Morgan} \EEE where only the case $\eta=1$ is addressed.  The case of different interaction intensities was already touched upon by {\sc Wecht, Barber, \& Tice} \cite{WBT}, who actually mention the three kinds of ensuing geometries and comment on the transition for small $\eta$, without exploring the full range $\eta\in(0,2)$ nor providing proof details. \EEE These generalizations are the product of a different proof 
strategy, \EEE 
based on a slicing argument for finite perimeter sets. Besides the above
mentioned generalization, our proof is not based on planar topology
and, we believe, shows prospects for extension to higher dimensions. 

Before closing this introduction, let us mention that the
double-bubble problem \eqref{eq:0double} is the continuous version of
the double-bubble problem on the square lattice. This discrete problem
has been recently tackled in \cite{Duncan,DB}. In particular, the case
of lattice sets of equal cardinality (volume) has been investigated in
\cite{DB} for all $\eta \geq 1$, showing basic geometrical features of
minimizers, presenting some specific minimizers, and  
discussing fluctuations estimates between different minimizers. 
The passage from discrete to continuum is at the core of \cite{Duncan}
where the distance between discrete minima and continuous ones is
explicitly quantified. Convergence of discrete minimizers to the
continuous ones in Figure~\ref{types} is discussed in \cite{DB}.

Compared with the continuum picture, the understanding of the discrete
case is still partial as it is restricted to specific ranges of interaction intensities
and volume ratios. Indeed, in the discrete case the problem
generically presents nonuniqueness of minimizers, posing
a substantial challenge to the analysis.

We recall some notation and state our main result Theorem
\ref{th:main} in Section \ref{sec:main}  where we also provide
some introductory discussion of the phase diagram of Figure \ref{diagram}. The proof of Theorem~\ref{th:main} is then given in Section \ref{sec:proof},  on the
basis of the discussion of the different parameter regimes. This
hinges on two propositions, exhausting all possible cases, which are proved in
Section \ref{sec:proof_propositions}. All technical lemmas needed in
the arguments are then proved in Section \ref{sec:proof_lemmas}.

\section{Main result}\label{sec:main}

In this section we present our main result  Theorem \ref{th:main}. Let us start by fixing some notation. We denote by $P_1(A)$ the $\ell_1$ perimeter of the finite
perimeter set $A\subset {\mathbb R}^2$, namely, 
$$P_1(A) = \int_{\partial^* A} (|\nu_1|+|\nu_2|)\, {\rm d}
\mathcal{H}^1,$$
where we recall that $\partial^*A$ stands for the {\it reduced
  boundary} of $A$, 
$\nu=(\nu_1,\nu_2)$ is the corresponding  (measure-theoretic) {\it outer
unit normal},  and $\mathcal H^1$ stands for the one-dimensional
Hausdorff measure \cite{Ambrosio-Fusco-Pallara}. Recall that 
 $\ell_1(F)$  indicates the {\it $\ell_1$-length} of a rectifiable subset $F$ of
$\partial^* A$  and is defined as 
$$\ell_1(F) =  \int_{F} (|\nu_1|+|\nu_2|)\, {\rm d}
\mathcal{H}^1$$
so that $\ell_1(\partial^* A) \equiv P_1(A)$. 

\subsection{Main result} 
The three configurations $(A,B)$ from Figure \ref{types} and their
corresponding energies are given as follows:

\begin{itemize}
  \item {\bf Type I:}
    \begin{equation}
      \label{eq:case1}
      A=[0,V_A/M]\times [0,M], \quad B=[-V_B/M,0]\times[0,M] \quad
      \text{with} \ \ M= \sqrt{\frac{2(V_A+V_B)}{2+\eta}},
    \end{equation}
    with corresponding energy $E(A,B) = 2\sqrt{4+2\eta}\sqrt{V_A + V_B}$.
 \item {\bf Type II:}  \ \ (only if  $V_B/V_A \leq \eta/2$) \EEE 
    \begin{equation}
      \label{eq:case2}
      A=[0,\sqrt{V_A}]^2, \quad B=[-V_B/M,0]\times[0,M] +(0,\lambda)\quad
      \text{with} \ \ M= \sqrt{\frac{2V_B}{\eta}}
    \end{equation}
     for some $\lambda\in[0,\sqrt{V_A}-M]$, \EEE
 with corresponding energy $E(A,B) = 4\sqrt{V_A} + 2\sqrt{2\eta V_B}$.
 \item {\bf Type III:}
    \begin{equation}
      \label{eq:case3}
      A=[0,\sqrt{V_A+V_B}]^2\setminus B, \quad B= [0,\sqrt{V_B}]^2,
    \end{equation}
     with corresponding energy $E(A,B) =  4\sqrt{V_A + V_B} + 2\eta \sqrt{V_B}$.
  \end{itemize}

%
%
%

We define the {\it phase-separation functions} by
\begin{align*}
  r_{12}(\eta) &= \frac{\eta}{2}, \\[1.5mm]
  r_{13}(\eta) &= \frac{8 + 2\eta - 4 \sqrt{4 + 2\eta}}{\eta^2 - 8 -
                 2\eta + 4 \sqrt{4 + 2\eta}}, \\
  r_{23}(\eta) &= \bigg( \frac{4\eta - 4 \sqrt{2\eta}}{2 \eta + \eta^2 - 2\eta \sqrt{2\eta} - 4} \bigg)^2. 
\end{align*}
Our main result shows that the graphs of the phase-separation functions partition the phase diagram into three regions, see Figure \ref{diagram} for an illustration. In particular, for $\eta = 1$, 
$$r_{12}(1) = \frac{1}{2}, \quad   r_{23}(1)  =  \left(\frac{4(\sqrt{2}-1)}{1+2\sqrt{2}}\right)^2 \sim
0.1872957155 $$
coincide with the critical values given in \cite{Duncan0}. The three types form a triple junction at 
$$\eta_{\rm triple} :=  1 + 2\sqrt{2} - \sqrt{5 + 4\sqrt{2}} \sim 0.56394,  \quad   r_{\rm triple} = \eta_{\rm triple}/2,$$
namely $r_{12}(\eta_{\rm triple}) = r_{13}(\eta_{\rm triple}) = r_{23}(\eta_{\rm triple}) $.

Our main result reads as follows.
 
\begin{theorem}[Characterization of \EEE minimizers]\label{th:main}
Let $(A,B)$ be an optimal configuration with volumes $V_A, V_B >0$
with $V_A\geq V_B$. Then, depending on $r= V_B/V_A \EEE \in (0,1]$ and
$\eta \in (0,2)$  and \EEE up to isometries  preserving \EEE the
coordinate axes we have that  
  
\begin{itemize}
\item   In case \EEE $\eta \in [\eta_{\rm triple},2)$,   \smallskip
  \begin{itemize}
  \item[] if $r > r_{12}(\eta)$, minimizers are of {\rm Type I}, \smallskip
    \item[] if
    $r_{23}(\eta) < r < r_{12}(\eta)$, minimizers are of {\rm Type
      II}, \smallskip
   \item[] if $r < r_{23}(\eta)$, minimizers are of {\rm Type
       III}. \smallskip
     \item[] if $r = r_{12}(\eta)$,  both {\rm Type I} and {\rm Type
         II} are optimal (and coincide), \EEE \smallskip
       \item[] if
  $r = r_{23}(\eta)$, both {\rm Type II} and {\rm Type III} are
  optimal. \smallskip
  \end{itemize} 

\item     In case   \EEE $\eta \in (0,\eta_{\rm triple})$, \smallskip
  \begin{itemize}
  \item[] if $r > r_{13}(\eta)$, minimizers are of {\rm Type I}, \smallskip
    \item[] if
      $r < r_{13}(\eta)$, minimizers are of {\rm Type III}, \smallskip
      \item [] if
  $r = r_{13}(\eta)$, both {\rm Type I} and {\rm Type III} are
  optimal.
  \end{itemize}

\end{itemize}

\end{theorem}

 
\subsection{Heuristics for the shape of the phase-separation functions}

In this subsection, we give  some justification \EEE  for the 
expressions of the phase-separation functions $r_{12}$, $r_{13}$, and
$r_{23}$. \EEE Let  $r= V_B/V_A \in (0,1]$, and  rewrite \EEE the
energy corresponding to the three types  as \EEE
\begin{align}\label{eq: ener1}
  E_{\rm I}(r,\eta) &= 2 \sqrt{4 + 2\eta} \sqrt{1+r}\sqrt{V_A},\\[1mm]
  \label{eq: ener2}
  E_{\rm II}(r,\eta) &=
                       \left\{
                       \begin{array}{ll}
                       (4 + 2 \sqrt{2r\eta})\sqrt{V_A} & \text{ if $r
                                                         \le
                                                         {\eta}/{2}$,}
  \\  + \infty & \text{ else,}
                       \end{array}  \right.
  \\[1mm]
  \label{eq: ener3}
E_{\rm III}(r,\eta) &= (4\sqrt{1+r} + 2\eta \sqrt{r})\sqrt{V_A},
\end{align}
see \eqref{eq:case1}, \eqref{eq:case2}, and \eqref{eq:case3}. Here, we  stress that
Type II  is defined for $r \le {\eta}/{2}$ only. We therefore set $E_{\rm II}(r,\eta) = + \infty$ if $r > {\eta}/{2}$.

By direct computation we now show that 
\begin{align}
 \label{I-II} E_{\rm I}(r,\eta) \le E_{\rm II}(r,\eta) \ &\Leftrightarrow \ r \ge
   r_{12}(\eta),\\
  \label{I-III} E_{\rm I}(r,\eta) \le E_{\rm III}(r,\eta) \ &\Leftrightarrow \ r \ge
    r_{13}(\eta),\\
  \label{II-III} E_{\rm II}(r,\eta) \le E_{\rm III}(r,\eta) \ &\Leftrightarrow \
  \eta \ge  \eta_{\rm triple} \ \text{and} \ r \ge r_{23}(\eta), 
\end{align}
i.e., the phase-separation functions partition the phase diagram into three regions corresponding to the respective optimal type.  

%




Let us first compare the energies corresponding to {\rm Types I} and
{\rm II}. To see this, we  investigate \EEE when the two energies
are equal,  getting \EEE
\begin{equation}
   E_{\rm I}(r,\eta) = E_{\rm II}(r,\eta) \ \Leftrightarrow \
  2 \sqrt{4 + 2\eta} \sqrt{1+r} = 4 + 2 \sqrt{2r\eta}. \EEE
\end{equation}
 Elementary algebraic manipulations show that the latter holds if
and only if \EEE 
\begin{equation}
(\sqrt{\eta} - \sqrt{2r})^2 = 0,
\end{equation}
 namely, 
if and only if $r = r_{12}(\eta)= \EEE {\eta}/{2}$. \EEE As a
consequence, the sign of $E_{\rm I}(r,\eta) - E_{\rm II}(r,\eta)$ can
only change at $r = {\eta}/{2}$.  As  $E_{\rm I}(r,\eta) - E_{\rm
  II}(r,\eta)$ \EEE is positive at $r = 0$,  one has that  $E_{\rm
  I}(r,\eta) < E_{\rm II}(r,\eta)$ if and only if $r>
r_{12}(\eta)$. Note that, for $r=r_{12}(\eta)=\eta/2$ the two types
coincide. In particular, one has \eqref{I-II}. \EEE 

 We now \EEE  compare the energies for {\rm Types I} and {\rm
  III}.  One has that  \EEE 
\begin{equation}
 E_{\rm I}(r,\eta) = E_{\rm III}(r,\eta)  \ \Leftrightarrow \ 2 \sqrt{4
  + 2\eta} \sqrt{1+r} = 4\sqrt{1+r} + 2\eta \sqrt{r} .
\end{equation}
 The latter can be rewritten as \EEE
\begin{equation}
2\eta \sqrt{r} = (2 \sqrt{4 + 2\eta} - 4) \sqrt{1+r}.
\end{equation}
 This equation is uniquely solved by \EEE
\begin{equation}
r = \frac{8 + 2\eta - 4 \sqrt{4 + 2\eta}}{\eta^2 - 8 - 2\eta + 4
  \sqrt{4 + 2\eta}}  \equiv r_{13}(\eta). \EEE
\end{equation}
 The  function $r_{13}$ is monotone decreasing.  Its graph
intersects \EEE 
the line $r = {\eta}/{2}$ at 
\begin{equation}
r_{ \rm triple} = \frac{1 + 2\sqrt{2} - \sqrt{5 + 4\sqrt{2}}}{2} \sim 0.28197.
\end{equation}
Moreover,  one has that $r_{13}(\eta)\to 1/3$ as $\eta \to
0$. \EEE 
The sign of $E_{\rm I}(r,\eta) - E_{\rm
  III}(r,\eta)$ can only change at $r = r_{13}(\eta)$.  Since $E_{\rm I}(r,\eta) - E_{\rm
  III}(r,\eta)$ \EEE 
is positive at $r = 0$,  one has that
$E_{\rm I}(r,\eta) < E_{\rm
  III}(r,\eta)$ if and only if $r > r_{13}(\eta)$ and that the two
energies coincide for $r = r_{13}(\eta)$. In particular, \eqref{I-III}
holds.  \EEE 

Finally, we compare the energies for {\rm Types II} and {\rm
  III}.  Again, we \EEE
check the equality case, \EEE i.e.,
\begin{equation}
  E_{\rm II}(r,\eta) = E_{\rm III}(r,\eta)  \ \Leftrightarrow \
  4 + 2 \sqrt{2r\eta} = 4\sqrt{1+r} + 2\eta \sqrt{r}.
\end{equation}
 Solving the latter equation for $r$ we get   
\begin{equation}
r = \bigg( \frac{4\eta - 4 \sqrt{2\eta}}{2 \eta + \eta^2 - 2\eta
  \sqrt{2\eta} - 4} \bigg)^2  \equiv \EEE r_{23}(\eta). 
\end{equation}
 The
sign of $E_{\rm II}(r,\eta) - E_{\rm III}(r,\eta)$  changes at \EEE
$r = r_{23}(\eta)$  only. As  $E_{\rm II}(r,\eta) - E_{\rm
  III}(r,\eta)$ \EEE  is positive for small $r > 0$,   we have that
$E_{\rm II}(r,\eta) < E_{\rm III}(r,\eta)$ if and only if Type II is admissible, namely $r \leq \eta/2$, and \EEE $r > r_{23}(\eta)$. For $r = r_{23}(\eta)$ the two energies coincide,  so that \eqref{II-III} holds. \EEE


 Note that in the range $r\in [0,1]$ the graphs of $r_{12}$ and
$r_{23}$ cross at $(0,0)$ and $(\eta_{\rm triple},r_{\rm triple})$
only. As the graphs of $r_{12}$ and
$r_{13}$ cross at  $(\eta_{\rm triple},r_{\rm triple})$, as well, one
has that indeed $(\eta_{\rm triple},r_{\rm triple})$ is a triple
point. For this choice of parameters, all types of minimizers appear
(note however that Type I and Type II coincide). One can also check that $r'_{13}(\eta_{\rm
  triple})=r'_{23}(\eta_{\rm triple})$, see Figure \ref{diagram}.

By choosing $\eta=1$, this  covers the result in \cite{Duncan0}. In
particular, the critical value of $r$ separating Type II and Type III
from \cite{Duncan0} turns out to be exactly $r_{23}(1)$.

\section{Proof of Theorem \ref{th:main}}\label{sec:proof}

\subsection{Structure of the proof} \label{sec:structure}

The proof of Theorem \ref{th:main} is obtained as a result of a
cascade of intermediate statements, according to the diagram in Figure
\ref{proof}.
\pgfdeclareimage[width=0.9 \linewidth]{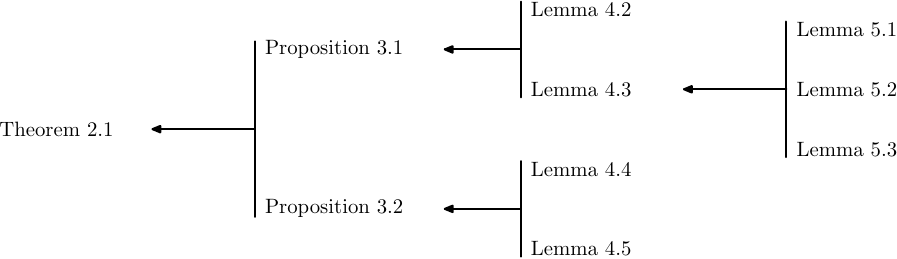}{proof}
\begin{figure}[h]
  \centering
  \pgfuseimage{proof}
\caption{Structure of the proof of Theorem \ref{th:main}.}\label{proof}
\end{figure}
In particular, Theorem \ref{th:main} is proved in this section via the
Propositions \ref{prop: Type1} and \ref{prop: Type2,3}. These propositions
are then proved in Section \ref{sec:proof_propositions}, also on the
basis of a series of lemmas. All lemmas are then proved in Section
\ref{sec:proof_lemmas}.

Note that the technical core of the proof is indeed represented by the
lemmas. In particular, Lemma \ref{lemma1} plays a pivotal role. Still, the combinatorial character of the construction makes the
argument quite involved. We hence resorted in recollecting the
findings in Propositions \ref{prop: Type1} and \ref{prop: Type2,3},
where the reader can directly check that all parameter regimes are
duly 
covered. 

\subsection{Two propositions}  As mentioned in Subsection
\ref{sec:structure}, Theorem \ref{th:main} can be proved from two
propositions, which we present in this  subsection. \EEE

Choose a coordinate direction,  to which we refer with no loss of
generality as {\it vertical},  and denote the length of the
projection of $A\cup B$ on the corresponding {\it
  horizontal} coordinate direction by $m$, as well as the length of the projection
of the two phases $A$ and $B$   by $m_A$ and $m_B$,
respectively.  (The terms {\it vertical} and {\it horizontal}
correspond to Figure \ref{types}.) As before, we let $r = \frac{V_B}{V_A}$.  \EEE Recall also the energies defined in \eqref{eq: ener1}--\eqref{eq: ener3}.

We state the following two propositions.

\begin{proposition}[{\rm Type I}]\label{prop: Type1}
For every configuration $(A,B)$, if one of the following conditions 

\begin{enumerate}
\item[\rm (a)] $r > r_{12}(\eta)$ and $r > r_{13}(\eta)$;

\item[\rm (b)] $r > r_{12}(\eta)$, $r = r_{13}(\eta)$ and there exists a coordinate direction such that
\begin{equation}\label{eq: length-B2,5}
m_B > \frac{2-\eta}{2\eta}\big(   (2+\eta)m_A - 2m   \big),
\end{equation}
\end{enumerate}

{\flushleft holds, \EEE then}
\begin{equation}\label{eq: energy-type1}
E(A,B) \ge E_{\rm I}(r,\eta)
\end{equation} 
and equality in \eqref{eq: energy-type1} implies that $(A,B)$ is a configuration of {\rm Type I}.
\end{proposition}

\begin{proposition}[{\rm Type II and Type III}]\label{prop: Type2,3}
For every configuration $(A,B)$, if $r\le \max\lbrace r_{12}(\eta), r_{13}(\eta)\rbrace$, we have
\begin{equation}\label{eq: energy-type2-3}
E(A,B) \ge  \min\big\{ E_{\rm II}(r,\eta), E_{\rm III}(r,\eta)\big\}. 
\end{equation}
The equality cases can be described more precisely as follows. Under the condition
\begin{enumerate}
\item[\rm (c)] $r \le r_{12}(\eta)$ and there exists a coordinate direction such that 
\begin{align}\label{eq: length-B1}
m_B > \frac{2-\eta}{2\eta}\big( (2+\eta)m_A - 2m \big) 
\end{align}
\end{enumerate}
{\flushleft we} have $E(A,B) \ge E_{\rm II}(r,\eta)$ and equality  implies that $(A,B)$ is a configuration of {\rm Type II}. Note that along the line $r = r_{12}(\eta)$ {\rm Type II} coincides with {\rm Type I}.

Furthermore, under the condition
\begin{enumerate}
\item[\rm (d)] $r\le \max\lbrace r_{12}(\eta), r_{13}(\eta)\rbrace$ and for both coordinate directions it holds that
\begin{align}\label{eq: length-B2}
m_B \le   \frac{2-\eta}{2\eta}\big(   (2+\eta)m_A - 2m   \big)
\end{align}
\end{enumerate}

{\flushleft we have that} $E(A,B) \ge E_{\rm III}(r,\eta)$ and equality   implies that $(A,B)$ is a configuration of {\rm Type III}.

Finally, the condition
\begin{enumerate}
\item[\rm (e)] $r > r_{12}(\eta)$,  $r < r_{13}(\eta)$   and  there
  exists a coordinate direction such that   
\begin{align}\label{eq: length-B12}
   m_B >   \frac{2-\eta}{2\eta}\big(   (2+\eta)m_A - 2m   \big), 
   \end{align}
\end{enumerate}

{\flushleft implies that} $E(A,B) > E_{\rm III}(r,\eta)$ and as a consequence $(A,B)$ cannot be a minimal configuration.
\end{proposition}

  As mentioned, the proofs of Propositions \ref{prop: Type1}
and \ref{prop: Type2,3} are presented in Section
\ref{sec:proof_propositions}.
\EEE

 The lower bounds \eqref{eq: energy-type1} and \eqref{eq:
  energy-type2-3}
 are the main point of the statements of Propositions \ref{prop:
  Type1} and \ref{prop: Type2,3}. \EEE Along the estimates, we will
also analyze when inequalities are strict. This then leads to the
characterization of the equality cases. Whereas Type I is treated on
its own, it is convenient to treat the lower bound of Type II and Type
III at the same time. The type of the ground state, namely Type II or
Type III, is then determined by the length of the projection of the
 $B$ \EEE phase, see \eqref{eq: length-B1}--\eqref{eq: length-B2}. \EEE

 We explicitly remark that Propositions  \ref{prop: Type1} and \ref{prop: Type2,3} cover
the parameter range $(r,\eta)\in (0,1)\times (0,2)$ entirely, by distinguishing
a number of cases. \EEE 
The situation is  illustrated \EEE in Figure
\ref{fig:propositions}. The portions  of the parameter space which
are far from boundaries \EEE 
are covered by cases {\rm (a), (c),} and {\rm (d)}  for Types I,
II, and III, \EEE respectively. The borderline case between {\rm Type
  I} and {\rm Type II} is covered by case {\rm (c)} (note that in this
case the two  types \EEE coincide  so that no phase transition actually
occurs). The \EEE  borderline case between {\rm Type I} and {\rm Type
  III} is covered by cases {\rm (b)} and {\rm (d)}.  The \EEE
borderline case between {\rm Type II} and {\rm Type III} \EEE is covered by
cases {\rm (c)} and {\rm (d)}  and \EEE the triple point is
covered by cases {\rm (c)} and {\rm (d)}.  Eventually, case {\rm
  (e)} is  proved  to \EEE be excluded as it does not correspond to a
minimizer. \EEE 

\begin{figure}[h]
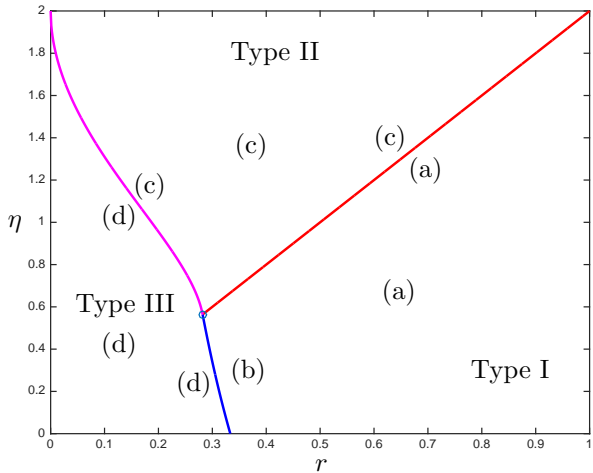

\centering 
\begin{overpic}[trim=0cm 75mm 0cm 7cm,
  width=100mm]{diagram.pdf}
  
\put(110,150){ Type II}
\put(200,30){ Type I}
\put(52,55){ Type III}

\put(115,115){(c)}
\put(170,60){(a)}
\put(65,40){(d)}

\put(167,118){(c)}
\put(180,106){(a)}

\put(77,100){(c)}
\put(65,88){(d)}

\put(110,30){ (b) \EEE}
\put(93,25){(d)}

\put(145,-5){$r$}
\put(30,87){$\eta$}
\end{overpic}

\vspace{0.4cm}

\caption{The three possible types from Figure \ref{types} and the corresponding cases in Propositions \ref{prop: Type1} and \ref{prop: Type2,3}.} 
\label{fig:propositions} 
\end{figure}

\EEE


 We now show that \EEE
  Propositions \ref{prop: Type1}--\ref{prop: Type2,3} imply Theorem \ref{th:main}. 

\begin{proof}[Proof of Theorem \ref{th:main}]
We first suppose that  $\eta \in [\eta_{\rm triple},2)$. 
If $r >  r_{12}(\eta)$,  noting that $r_{12}(\eta) \ge r_{13}(\eta)$ in this regime,   \EEE Proposition \ref{prop: Type1} implies that  minimizers are of {\rm Type I}. If $r \le  r_{12}(\eta)$, Proposition \ref{prop: Type2,3} yields that minimizers have energy $ \min\big\{ E_{\rm II}(r,\eta), E_{\rm III}(r,\eta)\big\}$. If $r_{23}(\eta) < r \le r_{12}(\eta)$, we get $ \min\big\{ E_{\rm II}(r,\eta), E_{\rm III}(r,\eta)\big\} = E_{\rm II}(r,\eta)$, see Figure~\ref{diagram}. Therefore, minimizers satisfy inequality \eqref{eq: length-B1} and are   of {\rm Type II}. Note that in the borderline case  $r =  r_{12}(\eta)$ minimizers are actually of {\rm Type I} since here {\rm Type I} and {\rm Type II} coincide. In a similar fashion, if $r < r_{23}(\eta)$, we have  $ \min\big\{ E_{\rm II}(r,\eta), E_{\rm III}(r,\eta)\big\} = E_{\rm III}(r,\eta)$, which implies that minimizers satisfy \eqref{eq: length-B2} and are of  {\rm Type III}. In the borderline case $r = r_{23}(\eta)$, a minimizer can be both of {\rm Type II} and {\rm Type III}, depending on whether \eqref{eq: length-B1} or \eqref{eq: length-B2} holds.  

 Let us now suppose that  $\eta \in (0,\eta_{\rm triple})$. 
 As before, by Proposition \ref{prop: Type1}, minimizers are of {\rm
   Type I} for $r >  r_{13}(\eta)$, noting that $r_{13}(\eta) \ge r_{12}(\eta)$ in this regime.   \EEE If $r < r_{13}(\eta)$,
 Proposition~\ref{prop: Type2,3} implies that minimizers have energy $
 \min\big\{ E_{\rm II}(r,\eta), E_{\rm III}(r,\eta)\big\}$, and then Figure \ref{diagram} gives $ 
 \min\big\{ E_{\rm II}(r,\eta), E_{\rm III}(r,\eta)\big\} = E_{\rm
   III}(r,\eta)$.  \EEE This yields that minimizers satisfy   \eqref{eq:
   length-B2} and are of {\rm Type III}. In the borderline case  $r =
 r_{13}(\eta)$, both {\rm Type I} and {\rm Type III} are optimal,
 depending on whether a minimizer satisfies \eqref{eq: length-B2,5} or
 \eqref{eq: length-B2}. This concludes the proof.
 \end{proof}

\section{Proof of Propositions \ref{prop: Type1} and \ref{prop: Type2,3}}\label{sec:proof_propositions} 

\subsection{Notation and  preliminaries\EEE} \label{sec: prel}

As before, we  denote  by $m$, $m_A$, and $m_B$ \EEE the length of the
 projection of $A\cup B$, $A$, and $B$, \EEE respectively. We
further set 
\begin{align}\label{p not}
p = \frac{m_A + m_B}{m} - 1.
\end{align}
We recall a slicing result for sets of finite perimeter, see for instance \cite[Section 18.3]{Maggi}. Suppose that $E$ is a bounded set of finite perimeter. We denote by $E_t$ the horizontal slice of the set $E$ at level $t$ in direction $x_2$, i.e.,
\begin{equation}
E_t = E \cap \{ (x_1,t): x_1 \in \mathbb{R} \}.
\end{equation}
Then, almost every horizontal slice $E_t$ is also a set of finite perimeter in $\mathbb{R}$, i.e., a union of intervals. 

%
%
%
%
%
%

We consider horizontal  slices  for
$t \in {\mathbb R}$ and set 
\begin{align}\label{eq: a,b}
a(t) := \mathcal{L}^1 (A_t ), \quad b(t) := \mathcal{L}^1 (B_t),
\end{align}
 where $\mathcal L^1$ stands for the one-dimensional Lebesgue
measure. \EEE
By Fubini's Theorem we have
$$V_A = \int_\R a(t) \, {\rm d}t, \quad \quad V_B =  \int_\R b(t) \, {\rm d}t. $$
 For convenience, we differentiate different slices as follows: \EEE we  define the  \textit{pure slices} \EEE $ \mathcal{T}^{\rm pure}_A \EEE = \{ t  \in
{\mathbb R} \EEE: a(t) > b(t) =
0\}$ and $ \mathcal{T}^{\rm pure}_B \EEE = \{  t  \in
{\mathbb R} \EEE:  \EEE b(t) > a(t) = 0 \}$, as well as the  \textit{mixed slices} 
\begin{align}\label{eq: T'}
 \mathcal{T}^{\rm mix} \EEE =   \{  t  \in
{\mathbb R} \EEE:  \EEE a(t) +  b(t) > 0 \} \EEE  \setminus ( \mathcal{T}^{\rm pure}_A \EEE \cup  \mathcal{T}^{\rm pure}_B \EEE).
\end{align}
For the respective areas, we introduce the notation
\begin{align}\label{eq: volume not}
   U^{\rm pure}_A  = \int_{  \mathcal{T}^{\rm pure}_A}  a(t) \, {\rm d}t  , \quad U^{\rm pure}_B  =  \int_{  \mathcal{T}^{\rm pure}_B}    b(t)\, {\rm d}t , \quad   U^{\rm mix} =   \int_{ \mathcal{T}^{\rm mix}}  (a(t) + b(t))\, {\rm d}t  . 
   \end{align} \EEE
We get $V_A + V_B =  U^{\rm pure}_A +  U^{\rm pure}_B + U^{\rm mix}$. 

 Denote by $E_{1D}(a(t),b(t))$ the minimal energy  in
one dimension \EEE 
for phases with  lengths \EEE $a(t)$ and $b(t)$,
respectively.  The minimal energy is attained if the horizontal slices  $A_t$, $B_t$ and $(A \cup B)_t$  are intervals (or empty), and in this case we have \EEE  
   \begin{equation}\label{eq: 1DDD}
  E_{1D}(a(t),b(t)) =
  \left\{
    \begin{array}{ll}
      2&\quad \text{on} \  \mathcal{T}^{\rm pure}_A \cup  \mathcal{T}^{\rm pure}_B,\\
           2+\eta&\quad \text{on} \  \mathcal{T}^{\rm mix}.
    \end{array}
  \right.\EEE
\end{equation} 

By a slicing argument  we obtain the following  first \EEE lower bound  on \EEE the energy of a configuration.

\begin{lemma}\label{lem:symmetrization-new}
Suppose that $A$ and $B$ are  bounded sets \EEE of finite perimeter. Then,
\begin{equation}\label{eq: part one slicing}
E(A,B)\geq  \int_\R E_{1D}(a(t),b(t)) \, {\rm d}t + m (2+   \eta   p)
\end{equation}
and
\begin{equation}\label{eq: part two slicing}
\int_{\partial^* (A  \cup B)}   |\nu \cdot e_2|\, {\rm  d} \EEE \mathcal{H}^{1} +   \eta \int_{\partial^* A  \cap \partial^* B}   |\nu \cdot e_2|\, {\rm  d} \EEE\mathcal{H}^{1} \geq m (2+  \eta  p ),
\end{equation}
where $\nu$ coincides with $\nu_A$ on $\partial^* A$ and with  $\nu_B$ on $\partial^* B \setminus \partial^* A$.
\end{lemma}

 The proof is standard and follows by slicing of sets of finite perimeter. We include a proof for convenience of the reader.

\begin{proof}
Let $E$ be a set of finite perimeter. By slicing properties of $BV$-functions, we obtain
\begin{equation}\label{eq:estimateofg2}
\int_{\partial^* E} |\nu_E \cdot e_1 | \, {\rm d} \mathcal{H}^{1} \EEE = \int_{\R}  \mathcal{H}^0\big(  \partial E_t \big) \, {\rm d}\mathcal{H}^{1}(t), 
\end{equation}
 where $\mathcal{H}^0$ is the counting measure. \EEE
Applying this projection argument on the sets $A$ and $B$ we find
\begin{align}\label{P1}
\int_{\partial^* (A\cup B)} |\nu \cdot e_1 | \, {\rm d} \mathcal{H}^{1} \EEE = \int_{\R}  \mathcal{H}^0\big(  \partial (A \cup B)_t \big) \, {\rm d}\mathcal{H}^{1}(t)
\end{align}
and 
\begin{align}\label{P2}\int_{\partial^* A \cap \partial^* B} |\nu \cdot e_1 | \, {\rm d} \mathcal{H}^{1} \EEE = \int_{\R}  \mathcal{H}^0\big(  \partial A_t \cap \partial B_t \big) \, {\rm d}\mathcal{H}^{1}(t)
\end{align}  
where $\nu$ coincides with $\nu_A$ on $\partial^* A$ and with  $\nu_B$ on $\partial^* B \setminus \partial^* A$.  In a similar fashion, we obtain 
\begin{align}\label{P3}
\int_{\partial^* (A\cup B)} |\nu \cdot e_2 | \, {\rm d} \mathcal{H}^{1} \EEE = \int_{\R}  \mathcal{H}^0\big(  \partial (A \cup B)^t \big) \, {\rm d}\mathcal{H}^{1}(t)
\end{align}
and 
\begin{align}\label{P4}
\int_{\partial^* A \cap \partial^* B} |\nu \cdot e_2 | \, {\rm d} \mathcal{H}^{1} \EEE = \int_{\R}  \mathcal{H}^0\big(  \partial A^t \cap \partial B^t \big) \, {\rm d}\mathcal{H}^{1}(t)
\end{align}  
where $E^t := E \cap \{ (t,x_2): x_2 \in \mathbb{R} \}$ for each bounded set of finite perimeter $E\subset \R^2$.

For each $t$ such that $\mathcal{L}^1((A \cup B)^t)>0$  it holds that
 $\mathcal{H}^0\big(  \partial (A \cup B)^t \big)  \EEE \ge
2$. \EEE Moreover, whenever $\mathcal{L}^1(A^t)>0$ and  $\mathcal{L}^1(B^t)>0$ we have
$$\mathcal{H}^0\big(  \partial A^t \cap \partial B^t) \ge 1 \quad
\quad \text{or} \quad \quad   \mathcal{H}^0\big(  \partial (A \cup B)^t \big)   \ge 4,$$
depending on whether the intersection of $\partial^*A$ and $\partial^*B$ on the slice is empty or not. Recalling the definition of $m$ and $p$, and setting $q := \frac{1}{m} \lbrace t\in \R \colon \mathcal{H}^0(\partial A^t \cap \partial B^t) \ge 1 \rbrace $, we have $q \le p$ and obtain  
\begin{align}\label{P5}
 \int_{\R}  \mathcal{H}^0\big(  \partial (A \cup B)^t \big) \, {\rm d}\mathcal{H}^{1}(t) \ge 2m + 2m(p-q), \quad \quad \quad \int_{\R}  \mathcal{H}^0\big( \partial A^t \cap \partial B^t \big) \, {\rm d}\mathcal{H}^{1}(t) \ge mq.
 \end{align}
By \eqref{P3}--\eqref{P4} this \EEE shows  \eqref{eq: part two slicing}, where we particularly use that $\min_{0\le q \le p} (2(p-q) + \eta q) = \eta p $ due to  $\eta <2$.  In a similar fashion, using $\partial^* A \cup \partial^* B = \partial^*(A \cup B) \cup (\partial^* A \cap \partial^*B)$ and  rewriting the energy defined  in \eqref{eq:0energy} equivalently as  $E(A,B) =  \ell_1(\partial^* (A\cup B)  )  + \eta \ell_1(\partial^* A \cap \partial^*B)$, the combination of \EEE \eqref{P1}--\eqref{P2} and  \eqref{P5} yields \EEE \eqref{eq: part one slicing}. 
\end{proof}

We introduce some more notation related to a different decomposition of slices. We set 
\begin{align}\label{eq: calT}
\hspace{-0.6cm} \mathcal{T}_0 = \lbrace t \in {\mathbb R} \EEE\colon r a(t) = b(t)>0\rbrace, \ \ \mathcal{T}_A = \lbrace t \in {\mathbb R} \EEE\colon r a(t) > b(t) \rbrace, \ \ \mathcal{T}_B = \lbrace t \in {\mathbb R} \EEE\colon r a(t) < b(t) \rbrace
\end{align}
and define, for convenience, 
\begin{align}\label{eq: UAB}
U_A = \int_{\mathcal{T}_A} (a(t) + b(t)) \, {\rm d}t, \quad U_B = \int_{\mathcal{T}_B} (a(t) + b(t)) \, {\rm d}t, \quad U_0 = \int_{\mathcal{T}_0} (a(t) + b(t)) \, {\rm d}t. 
\end{align}
Clearly, $V_A + V_B = U_A + U_B + U_0$.   We set
\begin{align}\label{eq: alohabeta}
\alpha(t):=\frac{b(t)}{a(t)} \in [0, r) \ \ \text{for} \ t \in \mathcal{T}_A \
\  \ \text{and} \  \ \ \beta(t):=\frac{a(t)}{b(t)}\in [0, 1/r) \ \ \text{for} \ t \in \mathcal{T}_B.
\end{align}
Since  $r\mathcal{L}^2(A) = rV_A = V_B =  \mathcal{L}^2(B)$, \EEE we have by the definition of $\mathcal{T}_0$
$$r \int_{\mathcal{T}_A\cup \mathcal{T}_B} a(t) \,   {\rm d}t =   \int_{\mathcal{T}_A\cup \mathcal{T}_B} b(t) \,   {\rm d}t $$
and therefore
\begin{align}\nonumber
\int_{\mathcal{T}_A} a(t) (r-\alpha(t)) \,{\rm d}t &=
                                                     \int_{\mathcal{T}_A}
                                                     (r a(t) - b(t))
                                                     \, {\rm d}t    =
                                                     \int_{\mathcal{T}_B}
                                                     (b(t) - r a(t))
                                                     \, {\rm d}t  \\
&=   \int_{\mathcal{T}_B} b(t) (1-r \beta(t)) \, {\rm d}t. \label{eq: V*} 
\end{align}

\subsection{Type I} 
In this subsection we prove Proposition \ref{prop: Type1}. We suppose that $r \ge  \max\lbrace r_{12}(\eta), r_{13}(\eta)\rbrace$, where  $r = V_B/V_A$. Recall \eqref{eq: T'}  and define
\begin{align}\label{p''''}
p' := \frac{ \mathcal{L}^1(\mathcal{T}^{\rm mix})}{\mathcal{L}^1(\mathcal{T}^{\rm pure}_A \cup  \mathcal{T}^{\rm pure}_B \cup \mathcal{T}^{\rm mix} )}, 
\end{align}
which is the analogous quantity  to \EEE  $p$ in \eqref{p not} for
the other coordinate direction.   Proposition \ref{prop: Type1}
\EEE follows from the following two lemmas,  which are then proved
in Section \ref{sec:proof_lemmas}. \EEE

\begin{lemma}\label{lem:psmaller13}
Let $r \ge \max\lbrace r_{12}(\eta), r_{13}(\eta)\rbrace$ and $(A,B)$
 be \EEE  a configuration.  For \EEE $p +p' >
2\frac{\sqrt{4+2\eta} -2 }{\eta}$  one has \EEE  $E(A,B) >E_{\rm I}(r,\eta)$.  In particular, if $E(A,B) \le E_{\rm I}(r,\eta)$,  it holds $\min \lbrace p,p' \rbrace \leq    1/2$.
\end{lemma}

\begin{lemma}\label{lemma1}
Let $r >r_{12}(\eta)$ and $r \ge r_{13}(\eta)$,  and let \EEE $(A,B)$ be a
\EEE configuration.  Suppose \EEE that we can choose a coordinate direction with $p \le 1/2$.  If $r = r_{13}(\eta)$, $m_A = m$, and  $\mathcal{L}^1(\mathcal{T}^{\rm pure}_B)=0$,  we additionally assume that $p'  >  \frac{2-\eta}{2 }\EEE$.    Then,  we have  
\begin{align}\label{eq: main estimate}
E(A,B) \geq (2+\eta p)m + \frac{2+\eta}{(1 + \eta\frac{p}{2})m} (V_A + V_B),
\end{align}
 and equality is attained if  and only if  $p = 0$, \EEE 
$\mathcal{L}^1(\mathcal{T}_A) = \mathcal{L}^1(\mathcal{T}_B) = 0$ and
$a(t) + b(t) = m$ for  $t \in \mathcal{T}_0\setminus \mathcal{N}$,
where  $\mathcal{L}^1(\mathcal{N})=0$. \EEE 
\end{lemma}

\begin{proof}[Proof of Proposition \ref{prop: Type1}]
Without restriction we can  assume that there is a coordinate direction with $p \le 1/2$ as otherwise by Lemma \ref{lem:psmaller13}  the estimate \eqref{eq: energy-type1} is strict. Let us further check that, if $r = r_{13}(\eta)$, $m_A = m$, and  $\mathcal{L}^1(\mathcal{T}^{\rm pure}_B)=0$,  we additionally have $p' >  \frac{2-\eta}{2 } \EEE$. Indeed, by $m=m_A$ we have $pm=m_B$, see \eqref{p not}. Then,  \eqref{eq: length-B2,5} implies that
$${\rm (i)}  \ \  p= \frac{m_B}{m} > \frac{2-\eta}{2 } \EEE \quad \quad  \text{ or } \quad \quad p'  >   \frac{2-\eta}{2 }.  $$
Since $r = r_{13}(\eta)$ and $r \ge r_{12}(\eta)$, we get $\eta \in
(0,\eta_{\rm triple}]$, see Figure \ref{diagram}, and thus
$\frac{2-\eta}{2} >\frac{1}{2}$. As $p \le 1/2$,  case (i) cannot
occur and case (ii) follows. \EEE This allows to apply Lemma \ref{lemma1}.
 
By the change of variables $M = (2+\eta p)m/(2+\eta)$, and
optimizing with respect to all possible values $m>0$ and $p \in
[0,1/2]$, Lemma \ref{lemma1} implies that  $E(A,B) \geq 2\sqrt{4+2\eta}\sqrt{V_A+V_B} = E_{\rm I}(r,\eta)\EEE$ with equality  if and
\EEE only if $p=0$, $\mathcal{L}^1(\mathcal{T}_A) =
\mathcal{L}^1(\mathcal{T}_B) = 0$, and $a(t) + b(t) = m$ for $t \in
\mathcal{T}_0 \setminus \mathcal{N}$. \EEE This allows to characterize
ground states. In fact, $p=0$ and $a(t) + b(t) = m$   yield that the projection of $A$ and $B$ have empty intersection and
each slice $t \in \mathcal{T}_0 \setminus \mathcal{N}$ has the same
geometry. This shows that minimizers are of {\rm Type I},  see \EEE \eqref{eq:case1}. 
\end{proof}

\subsection{Type II and Type III} 
In this subsection,   we prove Proposition \ref{prop: Type2,3} by
treating   configurations of \EEE Type II and Type III at once. We suppose that $r \ge 
 \bar r( \EEE\eta) := \EEE \max\lbrace r_{12}(\eta), r_{13}(\eta)\rbrace$, where  $r = V_B/V_A$. As before, we choose a coordinate direction and denote the length of the
projections by $m_A$,  $m_B$, and $m$, respectively. Recall also  the notations in \eqref{p not} and  \eqref{eq: T'}. The key ingredient of the proof  are the following two lemmas.

\begin{lemma}[Big projection]\label{lemma: biggi} The following two properties hold:\\
(Big projection $m_B$) Under the condition
  \begin{align}\label{eq: B1}
    m_B >   \frac{2-\eta}{2\eta}\big(   (2+\eta)m_A - 2m   \big),
  \end{align}
  we have the lower bound
\begin{align}\label{eq: lower3}
\hspace{-0.7cm} E(A,B)  \ge  2m_A + \eta\big(   m_B + (2\eta^{-1}-1)(m-m_A)  \big)  +  2\frac{V_A}{m_A} +  2\frac{V_B}{ m_B + (2\eta^{-1}-1)(m-m_A) }.
 \end{align}
   Equality can only hold if $a(t) = m_A$ for a.e.\ $t \in \mathcal{T}_A^{\rm pure} $ and $b(t) = m_B$ for a.e.\ $t \in \mathcal{T}^{\rm pure}_B $. Additionally, if $p>0$, equality can only hold if $\mathcal{L}^1(\mathcal{T}^{\rm mix} )= 0$.  \\
(Big projection $m_A$) Under the condition
  \begin{align}\label{eq: A1}
    m_A >   \frac{2-\eta}{2\eta}\big(   (2+\eta)m_B - 2m   \big),
  \end{align}
  we have the lower bound
  $$ E(A,B)  \ge  2m_B + \eta\big(   m_A + (2\eta^{-1}-1)(m-m_B)  \big)  +  2\frac{V_B}{m_B} +  2\frac{V_A}{ m_A + (2\eta^{-1}-1)(m-m_B) }.$$  
  \end{lemma}

\begin{lemma}[Small projection]\label{lemma: smalli}
The following two properties hold:\\
(Small projection $m_B$) Under the conditions
  \begin{align}\label{eq: B2}
   m_B \le  \frac{2-\eta}{2\eta}\big(   (2+\eta)m_A - 2m   \big), \quad \quad   \mathcal{ L \EEE}^1( \mathcal{T}^{\rm pure}_B \EEE \cup 
  \mathcal{T}^{\rm mix} \EEE) \le \Big(\frac{2}{\eta} -1 \Big) \mathcal{ L
  \EEE}^1(  \mathcal{T}^{\rm
  pure}_A ),
  \end{align}
  we have the lower bound
  \begin{align}\label{eq: lower1}
  E(A,B)  \geq  2m + \eta (m_A-m) + \eta m_B +  4\frac{V_A+V_B}{2m + \eta (m_A-m)} +  \eta \frac{V_B}{m_B}.
  \end{align}
     Equality can only hold if  $a(t) + b(t) = m$ and $b(t) = m_B$ for a.e.\ $t \in  \mathcal{T}^{\rm mix} $, and $a(t) = m_A$ for a.e.\ $t \in \mathcal{T}^{\rm pure}_A$.    Additionally, if $m_A  <m $ or $m_B < (2-\eta)m/2$, equality can only hold if $\mathcal{L}^1(\mathcal{T}^{\rm pure}_B )= 0$.  \\
(Small projection $m_A$)  Under the conditions 
   \begin{align}\label{eq: A2}
   m_A \le  \frac{2-\eta}{2\eta}\big(   (2+\eta)m_B - 2m   \big), \quad \quad m_B \ge \frac{2}{2+\eta}m, \quad \quad   \eta \le \eta_{\rm triple},  \quad r \le 1/3,
  \end{align} 
 we have the lower bound 
  \begin{align}\label{eq: lower2}
  E(A,B)  \geq  2m + \eta (m_B-m) + \eta m_A +  4\frac{V_A+V_B}{2m + \eta (m_B-m)} +  \eta \frac{2V_B}{m_A}.
  \end{align}
  \end{lemma}

We postpone the proofs of the lemmas to Section \ref{sec:proof_lemmas}, \EEE and continue with the second part of the proof of the main theorem.

\begin{proof}[Proof of Proposition \ref{prop: Type2,3}]
Let $r>0$ be such that $r\le  \bar r( \EEE\eta) = \max\lbrace r_{12}(\eta), r_{13}(\eta)\rbrace \EEE $. We split the proof into various steps.  We first show that one of the conditions  \eqref{eq: B1} or \eqref{eq: B2} is satisfied, where \eqref{eq: B1}  corresponds to Type II and  \eqref{eq: B2}  corresponds to Type III. The proof is then split into the cases  \eqref{eq: B1} and \eqref{eq: B2}, where for  \eqref{eq: B1} we distinguish further between  $r\le    \eta/2$  and   $r >   \eta/2$ to show that in the case $r >   \eta/2$ configurations are not minimizers.  (The latter corresponds to the fact that configurations of {\rm Type II} do not exist for $r >    \eta/2$.)

\noindent \emph{Step 1: Conditions \eqref{eq: B1} and \eqref{eq: B2}.}
We prove that  \eqref{eq: B1} or \eqref{eq: B2} is satisfied.  If
there exists  one coordinate direction such that the   corresponding
$m,m_A,m_B$ satisfy $ m_B >   \frac{2-\eta}{2\eta}(   (2+\eta)m_A -
2m)$, then \eqref{eq: B1} holds. Otherwise, for {\it both} coordinate directions, $m_B$ satisfies
$$  m_B \le  \frac{2-\eta}{2\eta}\big(   (2+\eta)m_A - 2m   \big).  $$
In this case, choosing  any
orientation with the  corresponding $m,m_A,m_B$  and letting $
\mathcal{T}^{\rm mix} \EEE$, $ \mathcal{T}^{\rm pure}_A \EEE$,
 and \EEE $ \mathcal{T}^{\rm pure}_B \EEE$ as in \eqref{eq: T'}, we have
\begin{align}
&m_B \le  \frac{2-\eta}{2\eta}\big(   (2+\eta)m_A - 2m   \big), \label{eq: mB1}\\
&\mathcal{ L \EEE}^1( \mathcal{T}^{\rm pure}_B \EEE \cup 
  \mathcal{T}^{\rm mix} \EEE) \le  \Big(\frac{2}{\eta} -1  \Big) \mathcal{ L  \EEE}^1(  \mathcal{T}^{\rm   pure}_A ). \label{eq: mB2}
\end{align}  
In fact, $\mathcal{ L \EEE}^1( \mathcal{T}^{\rm pure}_B \EEE \cup  \mathcal{T}^{\rm mix} \EEE) = \tilde{m}_B$,  $\mathcal{ L \EEE}^1( \mathcal{T}^{\rm pure}_A \EEE \cup  \mathcal{T}^{\rm mix} \EEE) = \tilde{m}_A$, and $\mathcal{ L
  \EEE}^1( \mathcal{T}^{\rm mix} \EEE \cup  \mathcal{T}^{\rm pure}_A \EEE \cup  \mathcal{T}^{\rm pure}_B \EEE ) = \tilde{m}$, where  $\tilde{m}_A$, $\tilde{m}_B$,  and \EEE $ \tilde{m}$ are associated to the other coordinate direction. By assumption we have  $ \tilde{m}_B \le    \frac{2-\eta}{2\eta}(   (2+\eta)\tilde{m}_A - 2\tilde{m})$ and, since $\tilde{m}_A \le \tilde{m}$, we get  $ \tilde{m}_B  \le \frac{2-\eta}{2} \tilde{m}$  , i.e.,
$$  \mathcal{ L \EEE}^1( \mathcal{T}^{\rm pure}_B \EEE \cup 
  \mathcal{T}^{\rm mix} \EEE) = \tilde{m}_B \le   \frac{2-\eta}{2} \tilde{m} =  \frac{2-\eta}{2} \mathcal{ L
  \EEE}^1( \mathcal{T}^{\rm mix} \EEE \cup  \mathcal{T}^{\rm
  pure}_A \EEE \cup  \mathcal{T}^{\rm pure}_B \EEE ).  $$
Now, this implies 
$$\mathcal{ L}^1( \mathcal{T}^{\rm pure}_B  \cup 
  \mathcal{T}^{\rm mix} ) \le \Big(1-\frac{\eta}{2} \Big) \mathcal{ L }^1( \mathcal{T}^{\rm pure}_B  \cup 
  \mathcal{T}^{\rm mix} ) + \Big(1-\frac{\eta}{2} \Big) \mathcal{ L }^1( \mathcal{T}^{\rm pure}_A)$$
which gives \eqref{eq: mB2}. This along with \eqref{eq: mB1} yields \eqref{eq: B2}. As the choice of the orientation is arbitrary, we can choose it such that the ratio of the projection of $A$ is maximized, i.e.,
  \begin{align}\label{eq: maximize}
\frac{m_A}{m} \ge \frac{\tilde{m}_A}{\tilde{m}} =  \frac{\mathcal{ L
  \EEE}^1( \mathcal{T}^{\rm mix} \EEE \cup  \mathcal{T}^{\rm pure}_A \EEE \cup  \mathcal{T}^{\rm pure}_B \EEE ) - \mathcal{ L
  \EEE}^1(  \mathcal{T}^{\rm pure}_B \EEE )}{ \mathcal{ L
  \EEE}^1( \mathcal{T}^{\rm mix} \EEE \cup  \mathcal{T}^{\rm pure}_A \EEE \cup  \mathcal{T}^{\rm pure}_B \EEE )}.  
  \end{align}

 \noindent \emph{Step 2: \eqref{eq: B2}.}  
 If  \eqref{eq: B2} holds, we have the lower bound \eqref{eq:
   lower1}. We now use the fact that for $x,y>0$ we have   $x + 4(V_A
 + V_B)/x \geq 4 \sqrt{V_A + V_B}$  and $y + V_B/y \geq 2 \sqrt{V_B}$
 with equality if and only if $ x =2 \sqrt{V_A + V_B}$ and $y =
 \sqrt{V_B}$  in order to \EEE conclude
\begin{align*}
E(A,B)  \geq  2m + \eta (m_A-m) + \eta m_B +  4\frac{V_A+V_B}{2m + \eta (m_A-m)} +  \eta \frac{V_B}{m_B}  \ge 4 \sqrt{V_A + V_B} + 2 \eta \sqrt{V_B}. 
\end{align*}
This shows the lower bound \eqref{eq: energy-type2-3}, more precisely that under \eqref{eq: length-B2} we have $E(A,B) \ge E_{\rm III}(r,\eta)$.

If equality holds, Lemma \ref{lemma: smalli} implies that   $a(t) + b(t) = m$ and $b(t) = m_B$ for a.e.\ $t \in  \mathcal{T}^{\rm mix} $, and $a(t) = m_A$ for a.e.\ $t \in \mathcal{T}^{\rm pure}_A$.  The second equality yields
\begin{align}\label{eq: lati}
 2m + \eta (m_A-m) = 2\sqrt{V_A+V_B}, \quad \quad \text{ and }\quad \quad m_B =\sqrt{V_B}.
 \end{align}
  If $(A,B)$ is a minimizer, we also have $r \le  r_{23}(\eta)$ or $\eta \le \eta_{\rm triple}$ as otherwise the  energy of Type II would be more convenient, see Figure \ref{diagram}. We now show that equality in the above estimate also implies $\mathcal{L}^1(\mathcal{T}^{\rm pure}_B )= 0$.
By Lemma \ref{lemma: smalli} it suffices to check that  $m_A < m$ or $m_B < (2-\eta)m/2$. Without restriction we suppose that $m_A = m$ and show $m_B < (2-\eta)m/2$. To this end, we observe that in this case $m = m_A = \sqrt{V_A+V_B} = \sqrt{V_B}\sqrt{1/r+1}$ and $m_B = \sqrt{V_B}$ hold. Now, one can check that $\frac{2-\eta}{2}\sqrt{1/r+1}>1$ in both cases $\lbrace (\eta,r)\colon \eta \in (0,2), r \le  r_{23}(\eta)\rbrace$ or $(\eta,r) \in (0,\eta_{\rm triple}] \times (0,1]$.


Now,  \eqref{eq: maximize} and  $\mathcal{L}^1(\mathcal{T}^{\rm
  pure}_B )= 0$ give $m_A=m$.  The properties  $a(t) + b(t) = m =m_A$
for a.e.\ $t \in \mathcal{T}_A^{\rm pure}  \cup \mathcal{T}^{\rm mix}
$,  $b(t) = m_B$ for a.e.\ $t \in \mathcal{T}^{\rm mix}$, and
$\mathcal{L}^1(\mathcal{T}^{\rm pure}_B )= 0$ imply that $A \cup B$ is
 a rectangle having  one side of \EEE length $m$, where $m= m_A$
satisfies $m = \sqrt{V_A+V_B}$, and $B$ is a rectangle with one 
side of length \EEE $m_B =\sqrt{V_B}$, see \eqref{eq: lati}.  This yields that $B$ is a square with  side \EEE $\sqrt{V_B}$  and $A \cup B$ is a square with  side \EEE $m = \sqrt{V_A + V_B}$, i.e., the optimal configuration $(A,B)$ is of {\rm Type III}.

 \noindent \emph{Step 3: \eqref{eq: B1} and  $r \le \eta /2$.}  
  If  \eqref{eq: B1} and  $r \leq \eta/2$ \EEE hold, we have the lower bound \eqref{eq: lower3}.   We now use the fact that for $x,y>0$ we have  $x + V_A/x \geq 2 \sqrt{V_A}$  and  $\eta y +
2V_B/y \geq 2\sqrt{2\eta} \sqrt{V_B}$ with equality if and only if $x
= \sqrt{V_A}$ and $y = \sqrt{2V_B/\eta}$  in order to \EEE conclude
\begin{align}
E(A,B)   &\geq 2m_A + \eta\big(   m_B + (2\eta^{-1}-1)(m-m_A)  \big)  +  2\frac{V_A}{m_A} +  2\frac{V_B}{ m_B + (2\eta^{-1}-1)(m-m_A) } \nonumber \\
&   \ge 4 \sqrt{V_A} + 2\sqrt{2\eta}\sqrt{V_B}. \label{eq: optimal1}
\end{align}
This shows the lower bound \eqref{eq: energy-type2-3}, more precisely that under \eqref{eq: length-B1} we have $E(A,B) \ge E_{\rm II}(r,\eta)$, see \eqref{eq: ener2}.

If equality holds, the first equality implies  $a(t) = m_A$ for a.e.\
$t \in \mathcal{T}_A^{\rm pure} $  and $b(t) = m_B$ for a.e.\ $t \in
\mathcal{T}^{\rm pure}_B $ by Lemma \ref{lemma: biggi}. If $p =0$, we
thus get that $A$ and $B$ are two rectangles attached to each other
with vertical interface. If $p>0$, equality and Lemma \ref{lemma:
  biggi} also imply that  $\mathcal{L}^1(\mathcal{T}^{\rm mix} )= 0$,
and thus $A$ and $B$ are two rectangles attached to each other with
horizontal interface. Up to changing coordinates, we can assume
without restriction that the latter \EEE arrangement occurs. Then, the
   rectangles have sides of lengths \EEE $m_A = \sqrt{V_A}$ and  $m_B = \sqrt{2V_B/\eta} -   (2\eta^{-1}-1)(m-m_A) $, respectively. We observe that $m_B \le \sqrt{V_A}$ since $V_B/V_A = r \le \eta /2$. This implies that $A$ is a square with  side \EEE   $\sqrt{V_A}$, $m=m_A$, and $B$ is a rectangle with  sides \EEE   $\sqrt{2V_B/\eta}$ and $ \sqrt{\eta V_B/2}$, i.e.,   optimal configurations $(A,B)$ are of {\rm Type II}.  Note that in the borderline case $r = \eta /2$ this corresponds to configurations of Type I.

 In Step 2 and Step 3, we have  established the desired lower bound for the energy and have characterized ground-state configurations.  We now need to exclude the remaining case in Step 4.

  \emph{Step 4: \eqref{eq: B1}  and   $r  > \eta/2$,
    $r<r_{13}(\eta)$.}  In \EEE this case, we show that
  configurations are not minimizers  by checking that the energy is
  strictly larger than the  energy of Type III. We distinguish three
   subcases: \EEE
$${\rm (a)} \ \ \max\lbrace m_A, m_B\rbrace \le \frac{2+\eta p}{2+\eta} m, \quad \quad  {\rm (b)} \ \    m_A > \frac{2+\eta p}{2+\eta} m, \quad \quad  {\rm (c)}  \ \  m_B > \frac{2+\eta p}{2+\eta} m.   $$

 Subcase (a): \EEE We use Lemma \ref{lemma1} and  the
forthcoming \EEE Remark \ref{rem: for later} to show that \eqref{eq: main estimate} holds. Optimizing in $p$ and $m$ we thus get 
$$E(A,B) \ge  2\sqrt{4+2\eta}\sqrt{V_A + V_B},   $$
corresponding  to the energy of  {\rm Type I}. see \eqref{eq: ener1}.  Since $r < r_{13}(\eta)$,  the energy of {\rm Type III} given in \eqref{eq: ener3} is optimal, which implies
\begin{align}\label{eq: 311}
E(A,B) \ge 2\sqrt{4+2\eta}\sqrt{V_A + V_B} >  4\sqrt{V_A + V_B} + 2\eta \sqrt{V_B}.
\end{align}
This shows that $(A,B)$ is not a minimizer.

 Subcase (b):  \EEE We treat the case $m_A > \frac{2+\eta p}{2+\eta} m$. Use the relation $m(1+p) = m_A + m_B$, which after a short computation yields $m_B \le m + \frac{2}{\eta}(m_A-m)$. This  is equivalent to 
\begin{align}\label{eq: contraint}
 m_B + (2\eta^{-1}-1)(m-m_A) \le m_A.
\end{align}
Since \eqref{eq: B1} is satisfied, \eqref{eq: lower3} in  Lemma \ref{lemma: biggi} holds. We get
$$E(A,B)   \geq 2m_A + \eta\big(   m_B + (2\eta^{-1}-1)(m-m_A)  \big)  +  2\frac{V_A}{m_A} +  2\frac{V_B}{ m_B + (2\eta^{-1}-1)(m-m_A) }. $$
As in Step 3, we now  minimize with respect to the variables $x = m_A$ and $y =  m_B + (2\eta^{-1}-1)(m-m_A)$, but now under the constraint that $y \le x$, see \eqref{eq: contraint}. In Step 3, we have seen that the optimal values of the convex functions $x\mapsto x + V_A/x $  and  $y \mapsto \eta y +
2V_B/y$ are $x = \sqrt{V_A}$ and $y = \sqrt{2V_B/\eta}$, respectively.   We distinguish the cases $m_A \le \sqrt{2V_B/\eta}$ and $m_A > \sqrt{2V_B/\eta}$.
 
 If $m_A \le \sqrt{2V_B/\eta}$, optimizing with respect to $y$ under the constraint $y \le x = m_A \le  \sqrt{2V_B/\eta}$ yields
\begin{align*}
E(A,B) &  \geq 2 x + \eta y  +  2\frac{V_A}{x} +  2\frac{V_B}{ y }  \ge    2 x + \eta x  +  2\frac{V_A}{x} +  2\frac{V_B}{ x }.
\end{align*}
Optimizing with respect to $x$ leads to $x =   \sqrt{2(V_A + V_B)/(2+\eta)}$ and thus,
\begin{align}\label{eq: the easy case}
E(A,B) \ge 2\sqrt{4+2\eta}\sqrt{V_A + V_B}.
\end{align}
As in \eqref{eq: 311}, this yields a strict lower bound to the  energy
of Type III,  showing that $(A,B)$ is not a minimizer. \EEE

Now, consider $m_A > \sqrt{2V_B\eta}$. As $r > \eta /2$, this implies
$m_A> \sqrt{2rV_A/\eta} > \sqrt{V_A}$, i.e., $m_A$ is  larger \EEE than the optimal value $x = \sqrt{V_A}$. This allows to estimate the energy from below by 
$$E(A,B)   \geq 2 \sqrt{2rV_A/\eta} + \eta y  +  2\frac{V_A}{\sqrt{2rV_A/\eta}} +  2\frac{V_B}{ y }.  $$
Optimizing with respect to $y$, we find  $y = \sqrt{2V_B/\eta}$ and thus
$$E(A,B)   \ge \sqrt{2V_A} \big(  2 \sqrt{r/\eta}   +  \sqrt{\eta/r} \big) +   2\sqrt{2\eta V_B} = \sqrt{2V_A} \big(  2 \sqrt{r/\eta}   +  \sqrt{\eta/r} + 2\sqrt{\eta r} \big).  $$
To conclude the proof, we need to show that 
$$ \sqrt{2V_A} \big(  2 \sqrt{r/\eta}   +  \sqrt{\eta/r} + 2\sqrt{\eta r} \big) > \sqrt{V_A}\big(   4 \sqrt{1+r}  + 2\eta \sqrt{r} \big)      =    4\sqrt{V_A + V_B} + 2\eta \sqrt{V_B}.$$
Dividing by $\sqrt{V_A r}$ this is equivalent to check  that
$$H_\eta(r) := \sqrt{2} \big(  2 /\sqrt{\eta}   +  \sqrt{\eta}/r + 2\sqrt{\eta} \big)  -   \big(   4 \sqrt{\tfrac{1}{r}+1}  + 2\eta  \big)  > 0.  $$
We compute
$$\frac{{\rm d}}{ {\rm d}r}H_\eta(r) = - \frac{\sqrt{2\eta}}{r^2} + 2\frac{1}{r^2\sqrt{1+1/r}} $$
to see that the function $H_\eta$ attains its minimum at $r =
\frac{\eta}{2-\eta}$  and is decreasing in $[0,
\frac{\eta}{2-\eta}]$. \EEE Now, we can directly check that
$$H_\eta(r)  \ge H_\eta(\eta/(2-\eta)) > 0  $$
for all $\eta \in (0, 1/2)$. If $\eta \in [1/2,\eta_{\rm triple})$, we
can calculate $r_{13}(\eta) < \frac{\eta}{2-\eta}$ and thus
$H_\eta(r)\geq H_\eta(r_{13}(\eta))$. On the other hand, 
                                %
one can check that  
$$ H_\eta(r_{13}(\eta)) > 0  $$
for all $\eta \in [1/2,\eta_{\rm triple})$. (Note that
$H_\eta(r_{13}(\eta)) = 0$ for $\eta = \eta_{\rm triple}$.)
Summarizing, also in this case the Type III energy is a strict lower
bound  and $(A,B)$ is not a minimizer. \EEE




 Subcase (c):  \EEE We treat the case $m_B > \frac{2+\eta p}{2+\eta} m$.  Use the relation $m(1+p) = m_A + m_B$, which, in a similar fashion as before, yields $m_A \le m + \frac{2}{\eta}(m_B-m)$. This  is equivalent to 
\begin{align}\label{eq: contraintXXX}
 m_A + (2\eta^{-1}-1)(m-m_B) \le m_B.
\end{align}
We further distinguish the  subcases \EEE
$${\rm (c1)} \ \    m_A \le  \frac{2-\eta}{2\eta}\big(   (2+\eta)m_B - 2m   \big), \quad \quad  {\rm (c2)} \ \    m_A >  \frac{2-\eta}{2\eta}\big(   (2+\eta)m_B - 2m   \big). $$
First, we consider  the subcase \EEE (c1). As $r > \frac{\eta}{2}$ and  $r < r_{13}(\eta)$,  we get  $\eta < \eta_{\rm triple}$ and $r \le 1/3$, cf.~Figure  \ref{diagram}. Thus, \eqref{eq: A2} is satisfied. Then, using $x + 4(V_A + V_B)/x \geq 4 \sqrt{V_A + V_B}$ 
and  \EEE $y + 2V_B/y \geq 2 \sqrt{2V_B}$ for $x,y >0$ we \EEE conclude
  $$E(A,B)  \geq  2m + \eta (m_B-m) + \eta m_A +  4\frac{V_A+V_B}{2m + \eta (m_B-m)} +  \eta \frac{2V_B}{m_A} \ge  4 \sqrt{V_A + V_B} +  2 \eta \sqrt{2V_B}.$$
Comparing to \eqref{eq: ener3}, this is strictly large than the Type
III energy  and $(A,B)$ is not a minimizer. \EEE
  
 Consider now the subcase \EEE (c2).  As \eqref{eq: A1}  is satisfied, Lemma \ref{lemma: biggi} yields 
  $$ E(A,B)  \ge  2m_B + \eta\big(   m_A + (2\eta^{-1}-1)(m-m_B)  \big)  +  2\frac{V_B}{m_B} +  2\frac{V_A}{ m_A + (2\eta^{-1}-1)(m-m_B) }.$$ 
We  minimize with respect to the variables $x = m_A + (2\eta^{-1}-1)(m-m_B)$ and $y =  m_B $, but now under the constraint $x \le y$, see \eqref{eq: contraintXXX}. One can check that optimal values for an unconstrained minimization   are $x =\sqrt{2V_A/\eta}$ and $y = \sqrt{V_B}$, respectively.   We distinguish the cases $y = m_B \le \sqrt{2V_A/\eta}$ and $m_B > \sqrt{2V_A/\eta}$.
 
 If $y  \le \sqrt{2V_A/\eta}$, we get $x \le \sqrt{2V_A/\eta}$ and  optimizing with respect to $x$ under the constraint $x \le y$ yields
\begin{align*}
E(A,B) &  \geq 2 y + \eta x  +  2\frac{V_A}{y} +  2\frac{V_B}{ x }  \ge    2 y + \eta y  +  2\frac{V_A}{y} +  2\frac{V_B}{ y } \ge  2\sqrt{4+2\eta}\sqrt{V_A + V_B},
\end{align*}
where the last step follows by the argument in \eqref{eq: the easy case}.   As in \eqref{eq: 311}, this  shows that $(A,B)$ is not a minimizer.

Now, consider $y > \sqrt{2V_A/\eta}$. As $V_A > V_B$ and $\eta < 2$, we clearly have $y > \sqrt{V_B}$.   This allows to estimate the energy from below by 
$$E(A,B)   \geq   2 y + \eta x  +  2\frac{V_A}{y} +  2\frac{V_B}{ x }  \ge  2 \sqrt{2V_A/\eta} + \eta x  +  2\frac{V_A}{\sqrt{2V_A/\eta}} +  2\frac{V_B}{ x }  .  $$
Optimizing with respect to $x$, we find  $x = \sqrt{2V_B/\eta}$ and thus
$$E(A,B)   \geq \sqrt{2V_A} \big(  2 /\sqrt{\eta}   +  \sqrt{\eta} \big) +   2\sqrt{2\eta V_B} = \sqrt{2V_A} \big(  2 /\sqrt{\eta}   +  \sqrt{\eta} + 2\sqrt{\eta r} \big).  $$
To conclude the proof, we need to show that 
$$ \sqrt{2V_A} \big(  2 /\sqrt{\eta}   +  \sqrt{\eta} + 2\sqrt{\eta r} \big) > \sqrt{V_A}\big(   4 \sqrt{1+r}  + 2\eta \sqrt{r} \big)      =    4\sqrt{V_A + V_B} + 2\eta \sqrt{V_B}.$$
Dividing by $\sqrt{V_A }$ this is equivalent to check  that
$$ H_\eta(r) := \sqrt{2} \big(  2 /\sqrt{\eta}   +  \sqrt{\eta} + 2\sqrt{\eta r} \big)  - \big(   4 \sqrt{1+r}  + 2\eta \sqrt{r} \big)  > 0. $$
It is elementary to see that $H_\eta$ is decreasing in $r$, namely using $r> \eta/2$   we get 
$$  H'_\eta(r) = -\frac{2}{\sqrt{1 + r}}  - \frac{\eta}{\sqrt{r}} + \frac{\sqrt{2\eta}}{\sqrt{r}} < -\frac{2}{\sqrt{1 + r}} +2 -2\sqrt{r} \le 0.$$
 Since $r \le 1/3$, see Figure \ref{diagram}, we get
\begin{align*}
H_\eta(r) \ge H_\eta(1/3) =  \sqrt{2} \big(  2 /\sqrt{\eta}   +  \sqrt{\eta} + 2\sqrt{\eta /3} \big)  - \big(   8/\sqrt{3}  + 2\eta /\sqrt{3} \big)  > 0
\end{align*}
for all $\eta \in [0,\eta_{\rm triple}]$.
\end{proof}

\section{Proofs of the Lemmas}\label{sec:proof_lemmas}

\subsection{Proof of Lemma \ref{lem:psmaller13}}
For convenience, we denote by $(m_1,m_2)$ the length of the projections of $A \cup B$ on the two coordinate directions, by $(m_1^A,m_2^A)$ the length of projections of $A$, and by $(m_1^B,m_2^B)$ the length of projections of $B$. For $i = 1,2$ we denote $p_i = \frac{m_i^A + m_i^B}{m_i} - 1$. Note that with  the  notation  from Subsection \ref{sec: prel} we have $p=p_1$ and $p' = p_2$. 

Suppose   that $p_1+p_2 > 2\frac{\sqrt{4+2\eta} -2 }{\eta}$. Write $\bar{p} = \frac{\sqrt{4+2\eta} -2 }{\eta}$ for shorthand. By   the second part of Lemma \ref{lem:symmetrization-new} applied to both the first and second coordinate, and rewriting the energy as done below \eqref{P5}, \EEE  we get
\begin{align}
E(A,B) &= \int_{\partial^* (A  \cup  B)}  (|\nu \cdot e_1| +
         |\nu \cdot e_2|)  \, {\rm d}\mathcal{H}^{1} +
         \eta\int_{\partial^* A  \cap  \partial^*  B}  (|\nu \cdot e_1| +
         |\nu \cdot e_2|)  \, {\rm d}\mathcal{H}^{1} \EEE  \\ &\ge m_1 (2+\eta p_1) + m_2 (2+  \eta p_2)   > m_1 (2+\eta p_1) + m_2 (2+  \eta (2\bar{p} - p_1)). 
\end{align}
Using $m_1 m_2 \ge V_A + V_B$ and optimizing with respect to $m_1$ we get
$$E(A,B) > m_1 (2+\eta p_1) + \frac{V_A+V_B}{m_1} (2+  \eta (2\bar{p} - p_1)) = 2\sqrt{(2+\eta p_1)(2+  \eta (2\bar{p} - p_1))  (V_A+V_B)}.$$
Then optimizing with respect to $p_1$ we find $p_1 = \bar{p}$ and thus
$$E(A,B) > 2(2+\eta \bar{p})   \sqrt{  V_A+V_B} =  2 \sqrt{4+2\eta}   \sqrt{  V_A+V_B},  $$
where the  last identity follows from the definition of
$\bar{p}$. This shows  $E(A,B) > E_{\rm I}(r,\eta)$, see \eqref{eq:
  ener1}. In particular, if $E(A,B) \le  E_{\rm I}(r,\eta)$, we have
$p_1+p_2 \le 2\frac{\sqrt{4+2\eta} -2 }{\eta}$, and thus
$\min\lbrace p_1,p_2\rbrace \le \frac{\sqrt{4+2\eta} -2 }{\eta}$. As
$\eta \in (0,2)$, this shows  $\min\lbrace p_1,p_2\rbrace \le 1/2$.

\subsection{Proof of Lemma \ref{lemma1}} We now proceed with the proof of Lemma \ref{lemma1}. To this end, we
need three further technical results,  which are then proved in the
following Subsections \ref{sec: lem:appendix}-\ref{sec: lemma: letztes lemma} \EEE 
 We define    $f \colon [0,\infty)
\to \R$ as
\begin{align}\label{eq: f def}
f(x) = \twopartdef{2}{x = 0}{2+\eta}{x > 0.}
\end{align}
and note that 
\begin{equation}
  E_{1D}(a(t),b(t)) =
  \left\{
    \begin{array}{ll}
      f(\alpha(t))&\quad \text{on} \ \mathcal{T}_A,\\
      f(\beta(t)) &\quad \text{on} \ \mathcal{T}_B,\\
      2+\eta&\quad \text{on} \ \mathcal{T}_0,
    \end{array}
  \right. 
\end{equation} 
where $E_{1D}$ is given in \eqref{eq: 1DDD}, and  $\alpha$ and $\beta$ stand for  the ratios defined on $\mathcal{T}_A$ and $\mathcal{T}_B$, see \eqref{eq: calT}--\eqref{eq: alohabeta}. \EEE  For $r \in [1/2,1]$ and
$p \in [0,{1}/{2}]$, we  define the auxiliary functions  $g_A^r,
\, g_B^r : [0,\infty) \to \mathbb{R}$ \EEE as 
\begin{align}
g_A^r(\alpha) & = \frac{1+\alpha}{r-\alpha} \Big(\frac{f(\alpha)}{ \min\lbrace m, (1+\alpha) m_A \rbrace   } - \frac{2+\eta}{(1 + \eta\frac{p}{2}) m} \Big),   \label{eq:gagb}\\
g_B^r(\beta) & = \frac{1+\beta}{1-r\beta} \Big(\frac{f(\beta)}{ \min\lbrace m, (1+\beta) m_B \rbrace} - \frac{2+\eta}{(1 +\eta \frac{p}{2}) m} \Big) \label{eq:gagb2}
\end{align}

\begin{lemma}\label{lem:appendix}
Let $r \ge \max \lbrace r_{12}(\eta),r_{13}(\eta)\rbrace$   and $p \in [0,1/2]$.  Suppose that 
\begin{equation}\label{eq:assumptionsonmamb}
m_A \geq \frac{2+\eta p}{2+\eta} m > \frac{\eta+2p}{2+\eta} m \geq m_B.   
\end{equation}
Then, we have
\begin{align}\label{eq:appendix}
 {\rm (a)}  \ \ \min_{\alpha \in [0,r]} (g_A^r(\alpha)) =  g_A^r(0) \le 0; \EEE  \quad \quad  {\rm (b)}  \ \ \min_{\beta \in [0,1/r]} g_B^r(\beta) =  \min \bigg\{ g_B^r(0),   g_B^r \bigg(\frac{m}{m_B} - 1 \bigg) \bigg\}  \ge 0. \EEE
\end{align}
The minimum of $g_B^r$ is attained at $\frac{m}{m_B} - 1$ if and only if $\frac{m_B}{m} \ge  \frac{r(2+\eta p)}{2+2r}$. Similarly, if
\begin{equation}\label{eq:assumptionsonmambprime}
m_B \geq \frac{2+\eta p}{2+\eta} m > \frac{\eta+2p}{2+\eta} m \geq m_A,    
\end{equation}
we have
\begin{align}
{\rm (c)} \ \ \min_{\alpha \in [0,r]} g_A^r(\alpha) =  \min \bigg\{ g_A^r(0),   g_A^r \bigg(\frac{m}{m_A} - 1 \bigg) \bigg\}  \ge 0;  \quad \quad  
{\rm (d)} \ \ \min_{\beta \in [0,1/r]} (g_B^r(\beta)) =  g_B^r(0) \le 0.
\end{align}
The minimum of $g_A^r$ is attained at $\frac{m}{m_A} - 1$ if and only if $\frac{m_A}{m} \ge  \frac{2+\eta p}{2+2r}$.  \EEE
\end{lemma}

\begin{lemma}\label{lemma: extra}
 Let $r >   r_{12}(\eta)$, $r \ge r_{13}(\eta)$, and $p \in [0,1/2]$.

{\flushleft \rm (a)} It holds that $g_B^r(0) + g_A^r(0) > 0$. 

{\flushleft \rm  (b)} If $\frac{\eta+2p}{2+\eta} \ge \frac{m_B}{m} \ge  \frac{r(2+\eta p)}{2+2r}$, we get  $g_B^r(\frac{m}{m_B} - 1) +  g^r_A(0) \ge 0$. If $\max \lbrace 1 - \frac{m_A}{m}, r - r_{13}(\eta)\rbrace >0$, the inequality is strict.  

{ {\flushleft \rm  (c)} If $ \frac{\eta+2p}{2+\eta} \EEE \ge \frac{m_A}{m} \ge \frac{2+\eta p}{2+2r}$ and $m_B > m_A$, \EEE we get $g_A^r(\frac{m}{m_A} - 1) +  g^r_B(0) > 0$.}
\end{lemma}

\begin{lemma}\label{lemma: letztes lemma}
Suppose that $(A,B)$ satisfies $E(A,B) \le E_{\rm I}(r,\eta)$.  If $r
= r_{13}(\eta)$,  $m_A = m$,   and $p' >   \frac{2-\eta}{2}$, 
with $p'$ defined in \eqref{p''''}, \EEE then $ \frac{m_B}{m} <
\frac{r(2+\eta p)}{2+2r}$.  
\end{lemma}

\begin{proof}[Proof of Lemma \ref{lemma1}]
We distinguish two cases, namely $\max\lbrace m_A,m_B \rbrace < \frac{2+\eta p}{2+\eta}m$ (Step 1) and $\max\lbrace m_A,m_B \rbrace \ge  \frac{2+\eta p}{2+\eta}m$ (Step 2).

 \noindent \emph{Step 1.} Suppose that $\max\lbrace m_A,m_B \rbrace < \frac{2+\eta p}{2+\eta}m$. By \eqref{eq: 1DDD} and \eqref{eq: part one slicing}  we have
\begin{align}\label{eq: anoth}
E(A,B) \geq (2+\eta p)m +   \int_{ \mathcal{T}^{\rm pure}_A} 2 \, {\rm d}t +  \int_{ \mathcal{T}^{\rm pure}_B} 2 \, {\rm d}t + \int_{ \mathcal{T}^{\rm mix}} (2+\eta) \, {\rm d}t.
\end{align}
We first estimate the last addend:   recalling \eqref{eq: volume not} and the fact  \EEE that by definition we have $a(t) + b(t) \leq m$ for all $t \in \mathbb{R}$, we get
\begin{equation}\label{eq:basicestimateint0}
\int_{ \mathcal{T}^{\rm mix} \EEE}  (2+\eta) \, {\rm d}t \geq \int_{ \mathcal{T}^{\rm mix} \EEE} \frac{ 2+\eta}{m} (a(t) + b(t))\, {\rm d}t = \frac{ 2+\eta}{m}  U^{\rm mix} \EEE,
\end{equation}
with strict inequality if $a(t) + b(t) < m$ on a subset of $ \mathcal{T}^{\rm mix} \EEE$ of positive $\mathcal{L}^1$-measure. Similarly,  for \EEE  the second and third addends we have
\begin{equation}
\int_{ \mathcal{T}^{\rm pure}_A \EEE} 2 \, {\rm d}t \geq \int_{ \mathcal{T}^{\rm pure}_A \EEE} \frac{2}{m_A} a(t) \, {\rm d}t = \frac{2}{m_A}  U^{\rm pure}_A \EEE
\end{equation}
and
\begin{equation}
\int_{ \mathcal{T}^{\rm pure}_B \EEE} 2 \, {\rm d}t \geq \int_{ \mathcal{T}^{\rm pure}_B \EEE} \frac{2}{m_B} b(t) \, {\rm d}t = \frac{2}{m_B}  U^{\rm pure}_B \EEE.
\end{equation}
In each of  these terms, we  \EEE subtract \EEE $\frac{2+\eta}{(1+\eta\frac{p}{2})m}$ under the integral. For $ \mathcal{T}^{\rm pure}_A \EEE$, we get
\begin{equation}\label{eq: anoth2}
\int_{ \mathcal{T}^{\rm pure}_A} 2 \, {\rm d}t \geq \left(\frac{2}{m_A} - \frac{2+\eta}{(1+\eta\frac{p}{2})m}\right)  U^{\rm pure}_A + \frac{2+\eta}{(1+\eta\frac{p}{2})m}  U^{\rm pure}_A.
\end{equation}
Due to the assumption  $\max\lbrace m_A,m_B \rbrace < \frac{2+\eta
  p}{2+\eta}m$,  the first summand on the  above \EEE right-hand
side is nonnegative and positive if $\mathcal{T}^{\rm pure}_A$ has
positive measure,   i.e.,  if we have $ \mathcal{L}^1(U_A^{\rm
  pure}) \EEE > 0$.  \EEE A  similar computation holds for $B$. Thus, combining \eqref{eq: anoth}, \eqref{eq:basicestimateint0}, \eqref{eq: anoth2} we obtain
\begin{equation}\label{eq: anoth3}
E(A,B) \ge \frac{2+\eta}{(1+\eta\frac{p}{2})m}  U^{\rm pure}_A + \frac{2+\eta}{(1+\eta\frac{p}{2})m}  U^{\rm pure}_B + \frac{ 2+\eta}{m}  U^{\rm mix}.
\end{equation}
This, along with $U^{\rm pure}_A + U^{\rm pure}_B + U^{\rm mix} =
V_A+V_B$, shows the lower bound \eqref{eq: main estimate}. If
$\mathcal{T}^{\rm pure}_A$ or $\mathcal{T}^{\rm pure}_B$  has positive
measure or  $a(t) + b(t) < m$ on a subset of $\mathcal{T}^{\rm mix}$
of  positive measure, we get  a \EEE strict inequality in
\eqref{eq: main estimate}.  Suppose now
$\mathcal{L}^1(\mathcal{T}^{\rm pure}_A) =
\mathcal{L}^1(\mathcal{T}^{\rm pure}_B) = 0 $ and $a(t) + b(t) = m$
for a.e.\ $t \in \mathcal{T}^{\rm mix}$. Then, $U^{\rm mix} = V_A +
V_B$. If $p>0$,  \eqref{eq: main estimate} is again strict by
\eqref{eq: anoth3}.     Along with \EEE $a(t) + b(t) = m$ for
a.e.\ $t \in \mathcal{T}^{\rm mix}$,  this \EEE shows that $
\mathcal{T}^{\rm mix}  = \mathcal{T}_0$ and consequently
$\mathcal{L}^1(\mathcal{T}_A) = \mathcal{L}^1(\mathcal{T}_B) = 0$,
i.e., all required properties are satisfied  also in the equality case.   \EEE

 \noindent \emph{Step 2.} Now we suppose that $\max\lbrace m_A,m_B \rbrace \ge \frac{2+\eta p}{2+\eta}m$. By the definition of $f$ in \eqref{eq: f def}, \EEE \eqref{eq: part one slicing}, \eqref{eq: calT}, and \eqref{eq: alohabeta} \EEE we get
\begin{align}\label{eq:foursummands}
E(A,B) \geq (2+\eta p)m +   \int_{\mathcal{T}_A} f(\alpha(t))\, {\rm d}t +  \int_{\mathcal{T}_B} f(\beta(t))\, {\rm d}t + \int_{\mathcal{T}_0} (2+\eta) \, {\rm d}t. \end{align}
Using  $a(t) + b(t) \le m$ and $a(t) \le m_A$ for a.e.\ $t \in \R$,
and recalling the  definition \eqref{eq:gagb} of the function \EEE
$g_A^r$,  we can  control \EEE the $\mathcal{T}_A$-term in \EEE  estimate \eqref{eq:foursummands} as  
\begin{align}
\int_{\mathcal{T}_A} f(\alpha(t))\, {\rm d}t &\geq   \int_{\mathcal{T}_A}
  \frac{1}{\min(m,(1+\alpha(t)) m_A)} (a(t) + b(t)) f(\alpha(t))\, {\rm d}t \EEE \\ &\geq
  \int_{\mathcal{T}_A} a(t) (r-\alpha(t)) g_A^r( \alpha(t) \EEE) \, {\rm d}t + \frac{2+\eta}{(1 +  \eta\frac{p}{2}) m} U_A.
\end{align}
Here, we used also the definition of $U_A$ in \eqref{eq:
  UAB}. Similarly, we  bound \EEE 
the $\mathcal{T}_B$-term in \EEE  estimate \eqref{eq:foursummands} as
\begin{align}
\int_{\mathcal{T}_B} f(\beta(t))\, {\rm d}t &\geq   \int_{\mathcal{T}_B}
  \frac{1}{\min(m, (1 + \beta(t)) m_B)} (a(t) + b(t)) f(\beta(t))\, {\rm d}t \EEE \\ &\geq
  \int_{\mathcal{T}_B} b(t) (1-r\beta(t)) g_B^r( \beta(t) \EEE) \, {\rm d}t + \frac{2+\eta}{(1 + \eta \frac{p}{2})m} U_B.
\end{align}
 Inserting \EEE  these expressions \EEE into
\eqref{eq:foursummands}  and \EEE
using  an argument similar to \eqref{eq:basicestimateint0}  for the
$\mathcal{T}_0$-term we get \EEE
\begin{equation}\label{eq: for lateruse}
E(A,B) 
 \geq (2+\eta p)m +  \int_{\mathcal{T}_A} a(t) (r-\alpha(t)) g_A^r(
  \alpha(t) \EEE ) \, {\rm d}t   +   \int_{\mathcal{T}_B}  b(t)
  (1-r\beta(t)) g_B^r( \beta (t) \EEE) \, {\rm d}t + \hat{U},
\end{equation}
where we set 
\begin{align}\label{hat U}
\hat{U} := \frac{2+\eta}{(1 + \eta \frac{p}{2})m} (U_A + U_B) +   \frac{2+\eta}{m} U_0
\end{align}
for brevity.

From now on, we assume without loss of generality that $m_A \geq m_B$. The case $m_A < m_B$ follows among similar lines as we briefly explain at the end of the proof. As $m_B = (1+p)m - m_A$  and $m_A \ge (2+\eta p)m/(2+\eta)$ by assumption, we get $m_B \leq {(\eta+2p)m}/(2+\eta)$. Therefore, Lemma~\ref{lem:appendix}(a),(b) hold. \EEE 

Recall \eqref{eq: V*}  and \EEE denote  by $U_*$ the value
\begin{align}\label{eq: U*}
U_*:= \int_{\mathcal{T}_B} b(t)(1-r\beta(t)) \, {\rm d}t= 
\int_{\mathcal{T}_A} a(t)(r-\alpha(t)) \, {\rm d}t.
\end{align}
We claim that, whenever $U_*>0$, it holds that 
\begin{align}\label{e*}
e_* :=  \int_{\mathcal{T}_B}  b(t) (1-r\beta(t)) g_B^r(\beta(t)) \, {\rm   d}t   + \EEE   \int_{\mathcal{T}_A} a(t)(r-\alpha(t)) g_A^r(      \alpha(t)) \, {\rm d}t>0.
\end{align}
In fact, using that $r-\alpha \ge 0  $ on $\mathcal{T}_A$ and $1-r \beta \ge 0$ on $\mathcal{T}_B$  we get by Lemma \ref{lem:appendix}(a),(b) \EEE
\begin{align}
e_*  &\ge    \min \left\{ g_B^r(0),   g_B^r\left (\frac{m}{m_B} - 1\right) \right\}
       \int_{\mathcal{T}_B} b(t) (1-r\beta(t)) \EEE \,
       {\rm   d}t    + g_A^r(0)   \EEE \int_{\mathcal{T}_A} a(t)(r-\alpha(t)) \, {\rm d}t. \label{rhs}
\end{align}
If $\frac{m_B}{m} <  \frac{r(2+\eta p)}{2+2r}$, the right-hand side
 of \eqref{rhs} \EEE is positive by Lemma \ref{lem:appendix}, \EEE Lemma \ref{lemma: extra}(a), and $U_*>0$, see \eqref{eq: U*}. Now, suppose instead that
$\frac{m_B}{m} \ge  \frac{r(2+\eta p)}{2+2r}$. Without restriction we can assume that $g_B^r(0) >
g_B^r(\frac{m}{m_B} - 1)$ since otherwise we conclude as before by Lemma \ref{lemma: extra}(a). \EEE  If $r>r_{13}(\eta)$ or
$m_A < m$, the right-hand side  of \eqref{rhs} \EEE is positive by
Lemma \ref{lemma: extra}(b). If $r= r_{13}(\eta)$,  $m_A =
m$, and $\mathcal{L}^1(\mathcal{T}^{\rm pure}_B) >0$,  the fact that we assumed    $g_B^r(0) >
g_B^r(\frac{m}{m_B} - 1)$ implies that \EEE  the first inequality is strict.
It remains the case $r= r_{13}(\eta)$,  $m_A = m$, and
$\mathcal{L}^1(\mathcal{T}^{\rm pure}_B) =0$. Then by assumption we
have   $p'> \frac{2-\eta}{2}m $. Thus, Lemma \ref{lemma: letztes
  lemma} implies that $\frac{m_B}{m} <  \frac{r(2+\eta p)}{2+2r}$
which contradicts our assumption $\frac{m_B}{m} \ge  \frac{r(2+\eta p)}{2+2r}$. \EEE   This yields \eqref{e*}.

%

From \eqref{eq: for lateruse}  and \eqref{e*} we  have obtained \EEE that 
\begin{align}\label{eq: strict inequality2}
E(A,B) 
& \geq (2+\eta p)m +   \int_{\mathcal{T}_A} a(t) (r-\alpha(t))
          g_A^r( \alpha(t))  \, {\rm d}t   +  
          \int_{\mathcal{T}_B}  b(t) (1-r\beta(t)) g_B^r(
          \beta(t))   \, {\rm d}t + \hat{U} \\
& \geq (2+\eta p)m  +   \hat{U},
\end{align}
where the inequality is strict whenever $U_* >0$.
By \eqref{hat U} and  the fact that \EEE $U_A + U_B + U_0 = V_A + V_B$, the lower bound \eqref{eq: main estimate} holds.  To address the equality case, we hence suppose that 
\EEE  $U_* = 0$.  By the definition of $U_*$ this implies
$\mathcal{L}^1(\mathcal{T}_A) = \mathcal{L}^1(\mathcal{T}_B) =
0$. Moreover, if  $a(t) + b(t) < m$ on a subset of $\mathcal{T}_0$ of
positive measure, inequality \eqref{eq: for lateruse}  is \EEE
strict, and thus also inequality \eqref{eq: main estimate}  is \EEE strict. Eventually, since $\hat{U} =  \frac{2+\eta}{m} (V_A + V_B)$, \eqref{eq: main estimate} is strict unless $p=0$, which yields the desired result. \EEE

 Let us conclude by mentioning that  in case  $m_A < m_B$ one
proceeds \EEE similarly by interchanging the roles of $A$ and $B$, and employing Lemma \ref{lem:appendix}(c),(d) and Lemma \ref{lemma: extra}(c).
\end{proof}

\begin{remark}\label{rem: for later}\rm 
An inspection of the proof shows that in Step 1, i.e., for  $\max\lbrace m_A,m_B \rbrace < \frac{2+\eta p}{2+\eta}m$, the assumptions \EEE $p \le 1/2$   and $r \ge r_{13}(\eta)$  are \EEE  not needed. 
\end{remark}

%
%
%
%
%
%
%
%
%
%

\subsection{Proof of Lemma \ref{lemma: biggi}} 
We start with the first part of the statement, i.e., we assume   \eqref{eq: B1}. By \eqref{eq: T'}, \eqref{eq: part one slicing}, and $m_B = (1+p)m - m_A$ \EEE  we have
\begin{align}\notag
E(A,B)  &\geq (2 + \eta p)m  +  \int_{ \mathcal{T}^{\rm pure}_A \EEE} 2 \, {\rm d}t +  \int_{ \mathcal{T}^{\rm pure}_B \EEE} 2 \, {\rm d}t + \int_{ \mathcal{T}^{\rm mix} \EEE} (2+\eta) \, {\rm d}t\\
& = \EEE 2m_A + \eta\big(   m_B + (2\eta^{-1}-1)(m-m_A)  \big)  +  2\frac{V_A}{m_A} +  2\frac{V_B}{ m_B + (2\eta^{-1}-1)(m-m_A) }  \notag\\ 
& \ \ \  +  \int_{ \mathcal{T}^{\rm pure}_A \EEE} \left(2  - \frac{2a(t)}{m_A}\right)\,
                                                                                               {\rm d}t +  \int_{ \mathcal{T}^{\rm pure}_B \EEE} \left(2  - \frac{2b(t)}{ m_B + (2\eta^{-1}-1)(m-m_A) }\right) \, {\rm d}t \\
  & \ \ \ + \int_{ \mathcal{T}^{\rm mix} \EEE} \left((2+\eta) - \frac{2a(t)}{m_A} - \frac{2b(t)}{ m_B + (2\eta^{-1}-1)(m-m_A) }\right)  \, {\rm d}t. \label{eq: N5}
\end{align}
The first two integrals  in the above right-hand side \EEE are clearly nonnegative since $a(t) \le m_A$, $b(t)\le m_B$, and $(2\eta^{-1}-1)(m-m_A) \ge 0 $. Moreover, they are strictly positive if $m(t) < m_A$ or $b(t) < m_B$ on a set of positive measure, respectively.

 We now address the third integral  in the right-hand side of
 \eqref{eq: N5}. \EEE  Omitting  $ t$ dependencies, for brevity,
 \EEE   we first estimate
\begin{align*}
\frac{2a}{m_A} + \frac{2b}{m_B + (2\eta^{-1}-1)(m-m_A)} \le \frac{2a}{m_A} + \frac{2\min\lbrace m-a,m_B\rbrace}{m_B + (2\eta^{-1}-1)(m-m_A)}.
\end{align*} 
As the  expression on the right-hand side is linear with respect
to $a$ as long as $a \ge m-m_B$, \EEE the maximum is attained    \EEE at $a= m - m_B $ or 
at \EEE $a = m_A$. Thus, using also $m_A + m_B \ge m$, we get
\begin{align*}
\frac{2a}{m_A} + \frac{2b}{m_B + (2\eta^{-1}-1)(m-m_A)} \le \max  &\Big\{ \frac{2(m-m_B)}{m_A}  + \frac{2m_B}{m_B + (2\eta^{-1}-1)(m-m_A)}, \\ &   2 + \frac{2(m-m_A)}{m_B + (2\eta^{-1}-1)(m-m_A)} \Big\}. 
\end{align*}
We treat the two  options in the above maximum term \EEE separately. First, $m_A + m_B \ge m$  implies that $m_B + (2\eta^{-1}-1)(m-m_A) \ge 2(m-m_A)\eta^{-1}$ and thus
$$ 2 + \frac{2(m-m_A)}{m_B + (2\eta^{-1}-1)(m-m_A)} \le 2+\eta. $$
We also note that, if $p>0$, then $m_A + m_B >m$, and thus the above inequality is strict. For the other case, we get after some computation
\begin{align*}
 & \frac{2(m-m_B)}{m_A}  + \frac{2m_B}{m_B + (2\eta^{-1}-1)(m-m_A)} \\
  & \quad = -\frac{ (m_A +m_B-m) (m_A (\eta^2-4) +(4 -2\eta)m  + 2\eta
    m_B )  }{  m_A ( (\eta-2) (m_A  - m) + \eta m_B    )   }   + 2+\eta.
\end{align*}
We find that for  $ m_B > \frac{2-\eta}{2\eta}\big(   (2+\eta)m_A - 2m   \big)$ (see  \eqref{eq: B1}) the second factor  in the  numerator \EEE is positive. Combining both estimates we thus get 
 $$\frac{2a}{m_A} + \frac{2b}{m_B + (2\eta^{-1}-1)(m-m_A)} \le 2+\eta ,$$
i.e., the third integral  in \eqref{eq: N5} \EEE is
nonnegative. Moreover, we have seen that  such \EEE integral is positive whenever $\mathcal{L}^1(\mathcal{T}^{\rm mix})>0$ and $p>0$.
 
Therefore,  the sum of  all \EEE integrals   in \eqref{eq: N5}
\EEE is nonnegative, and the statement follows. The proof also shows that  equality can only hold if $a(t) = m_A$ for a.e.\ $t \in \mathcal{T}_A^{\rm pure} $ and $b(t) = m_B$ for a.e.\ $t \in \mathcal{T}^{\rm pure}_B $. Additionally if $p>0$, equality implies $\mathcal{L}^1(\mathcal{T}^{\rm mix})=0$.

Under condition \eqref{eq: A1},   the proof can be verbatim
repeated \EEE  by interchanging the roles of $m_A,V_A$ and $m_B,V_B$.

\subsection{Proof of Lemma \ref{lemma: smalli}}
We treat the two cases separately.

 \noindent \emph{Step 1.}  We suppose that \eqref{eq: B2} holds. Using  \eqref{eq: T'}, \eqref{eq: part one slicing}, and   $m_A + m_B = (1+p)m$ we have 
\begin{align}
E(A,B)  &\geq (2 + \eta p) m  +   \int_{ \mathcal{T}^{\rm pure}_A \EEE} 2 \, {\rm d}t +  \int_{ \mathcal{T}^{\rm pure}_B \EEE} 2 \, {\rm d}t + \int_{ \mathcal{T}^{\rm mix} \EEE} (2+\eta) \, {\rm d}t\notag\\
&\geq 2m + \eta (m_A-m) + \eta m_B +  4\frac{V_A+V_B}{2m + \eta (m_A-m)} +  \eta \frac{V_B}{m_B}\notag \\ 
& \ \ \  +  \int_{ \mathcal{T}^{\rm pure}_A \EEE} \left(2  - \frac{4a(t)}{2m + \eta (m_A-m)}\right)\, {\rm d}t
                                                                               +  \int_{ \mathcal{T}^{\rm pure}_B \EEE} \left(2  - \frac{4b(t)}{2m + \eta (m_A-m)} - \eta \frac{b(t)}{m_B}\right) \, {\rm d}t \\
  &\ \ \ + \int_{ \mathcal{T}^{\rm mix} \EEE} \left((2+\eta) - \frac{4(a+b)(t)}{2m + \eta (m_A-m)} -  \eta\frac{b(t)}{m_B}\right)  \, {\rm d}t. \label{eq: NNN}
\end{align}
We  separately \EEE consider the three integrals  above. The
first one can be controlled as follows:
\begin{align}\label{firsti integra}
\int_{ \mathcal{T}^{\rm pure}_A \EEE} \left(2  - \frac{4a(t)}{2m + \eta (m_A-m)}\right)\, {\rm d}t &\ge \mathcal{ L\EEE}^1( \mathcal{T}^{\rm pure}_A \EEE) \left( 2- \frac{4m_A}{2m + \eta (m_A-m)} \right) \notag \\
&=  2\mathcal{ L \EEE}^1( \mathcal{T}^{\rm pure}_A \EEE)  \frac{(2-\eta)(m-m_A)}{2m + \eta (m_A-m)}.
\end{align}
The  inequality  is  strict if $a(t) < m_A$ on a set of positive measure.    We can check that the second integral in the right-hand side of
\eqref{eq: NNN} is nonnegative,   i.e.,  
\begin{align}\label{second integra}
\int_{ \mathcal{T}^{\rm pure}_B \EEE} \left(2  - \frac{4b(t)}{2m + \eta (m_A-m)} - \eta \frac{b(t)}{m_B}\right) \, {\rm d}t \ge 0. 
\end{align}
 Indeed, let us start by proving that \EEE
\begin{align}\label{eq: claim}
m_B  \le \frac{(2-\eta)^2}{4} m + \frac{(2-\eta)\eta}{4}  m_A.
\end{align}
As  $m_B \le \frac{2-\eta}{2\eta}\big(   (2+\eta)m_A - 2m   \big)$, see \eqref{eq: B2}, we calculate by $m_A \le m$ \EEE
\begin{align*}
m_B &  \le  \frac{2-\eta}{2\eta}\big(   (2+\eta)m_A - 2m   \big) =   \frac{(2-\eta)^2}{4} m + \Big( -\frac{(2-\eta)^2}{4} - 2\frac{2-\eta}{2\eta} \Big)  m      + \frac{(2-\eta) (2+\eta)}{2\eta}   m_A\\
& \le   \frac{(2-\eta)^2}{4} m + \Big( -\frac{(2-\eta)^2}{4} - 2\frac{2-\eta}{2\eta} + \frac{(2-\eta) (2+\eta)}{2\eta}  \Big) m_A \\
& =  \frac{(2-\eta)^2}{4} m + \frac{(2-\eta)\eta}{4}  m_A.
\end{align*} 
Now, using \eqref{eq: claim}  at all $t$ we have that \EEE
$$\frac{4b(t)}{2m + \eta (m_A-m)} + \eta \frac{b(t)}{m_B} \le \frac{4m_B}{2m + \eta (m_A-m)} + \eta =  \frac{4m_B + \eta (2m + \eta (m_A-m)) }{2m + \eta (m_A-m)}\le
2  $$ 
 which  shows \eqref{second integra}. The proof also shows that the inequality in \eqref{eq: claim}  is strict if we have $m_A < m$ or $m_A = m$ and  $m_B< \frac{2-\eta}{2}m$. In this case, the second integral is strictly positive, provided that  $\mathcal{L}^1( \mathcal{T}^{\rm pure}_B)>0$.

For the third integral  in the right-hand side of \eqref{eq: NNN},
\EEE we calculate
\begin{align*}
 \int_{ \mathcal{T}^{\rm mix}} \left((2+\eta) - \frac{4(a+b)(t)}{2m + \eta (m_A-m)} -
 \eta \frac{b(t)}{m_B}\right)  \, {\rm d}t & \ge \mathcal{
                                      L\EEE}^1( \mathcal{T}^{\rm mix}) \left(
                                      (2+\eta) - \frac{4m}{2m + \eta (m_A-m)} - \eta \right)
                                       \\ & =  2\mathcal{ L\EEE}^1( \mathcal{T}^{\rm mix}) \frac{\eta(m_A-m)}{2m + \eta (m_A-m)},
\end{align*}
where the inequality is strict if  $a(t) + b(t) < m$ on a set of positive measure or $b(t) < m_B$ on a set of positive measure.  
Now, taking the sum of the three integrals  in \eqref{eq: NNN} gives 
\begin{align*}
&\int_{ \mathcal{T}^{\rm pure}_A } \left(2  - \frac{4a(t)}{2m + \eta (m_A-m)}\right)\, {\rm d}t  +  \int_{ \mathcal{T}^{\rm pure}_B} \left(2  - \frac{4b(t)}{2m + \eta (m_A-m)} - \eta\frac{b(t)}{m_B}\right) \, {\rm d}t \\
 & \ \ \  + \int_{ \mathcal{T}^{\rm mix}} \left((2+\eta) - \frac{4(a+b)(t)}{2m + \eta (m_A-m)} -  \eta \frac{b(t)}{m_B}\right)  \, {\rm d}t\\
& \ \    \ge  2\mathcal{ L\EEE}^1( \mathcal{T}^{\rm pure}_A)                                                                                                                                                                                                                                \frac{(2-\eta)(m-m_A)}{2m + \eta (m_A-m)} +  2\mathcal{ L\EEE}^1( \mathcal{T}^{\rm mix}) \frac{\eta(m_A-m)}{2m + \eta (m_A-m)} \\
  & \ \ \ge   2\mathcal{ L\EEE}^1( \mathcal{T}^{\rm pure}_A)                                                                                                                                                                                                                                \frac{(2-\eta)(m-m_A)}{2m + \eta (m_A-m)} +  2\mathcal{ L\EEE}^1( \mathcal{T}^{\rm mix} \cup \mathcal{T}^{\rm pure}_B ) \frac{\eta(m_A-m)}{2m + \eta (m_A-m)},
\end{align*}
where the first inequality is a consequence of the previous estimates and the second one follows from the fact that the coefficient $\frac{\eta(m_A-m)}{2m + \eta (m_A-m)}$ is not positive. Using the second part of \eqref{eq: B2} we get   $\mathcal{ L\EEE}^1( \mathcal{T}^{\rm mix} \cup  \mathcal{T}^{\rm pure}_B) \le \frac{2-\eta}{2}\tilde{m}$ and $ \mathcal{ L\EEE}^1( \mathcal{T}^{\rm pure}_A) \ge \frac{\eta}{2}\tilde{m}$ with $\tilde{m} = \mathcal{ L\EEE}^1( \mathcal{T}^{\rm mix} \cup  \mathcal{T}^{\rm pure}_A \cup  \mathcal{T}^{\rm pure}_B)$. Thus, the sum of the integrals can be bounded from below by
$$ 2\frac{(2-\eta)(m-m_A)}{2m + \eta (m_A-m)} \frac{\eta}{2} \tilde{m}  +  2  \frac{\eta(m_A-m)}{2m + \eta (m_A-m)} \frac{2-\eta}{2} \tilde{m} = 0.$$
In view of \eqref{eq: NNN}, this concludes the proof of the lower bound \eqref{eq: lower1}.  \EEE
The arguments also show that equality can only hold if  $a(t) + b(t) = m$ and $b(t) = m_B$ for a.e.\ $t \in  \mathcal{T}^{\rm mix} $, and $a(t) = m_A$ for a.e.\ $t \in \mathcal{T}^{\rm pure}_A$. Additionally, if $m_A  <m $ or $m_B < (2-\eta)m/2$, equality can only hold if $\mathcal{L}^1(\mathcal{T}^{\rm pure}_B )= 0$.

 \noindent \emph{Step 2.}  We now come to the second part of the statement and assume that \eqref{eq: A2} holds. Using again \eqref{eq: T'}, \eqref{eq: part one slicing}, and   $m_A + m_B = (1+p)m$  we have 
\begin{align}
E(A,B)  &\geq (2 + \eta p) m  +   \int_{ \mathcal{T}^{\rm pure}_A \EEE} 2 \, {\rm d}t +  \int_{ \mathcal{T}^{\rm pure}_B \EEE} 2 \, {\rm d}t + \int_{ \mathcal{T}^{\rm mix} \EEE} (2+\eta) \, {\rm d}t\notag\\
&\geq 2m + \eta (m_B-m) + \eta m_A +  4\frac{V_A+V_B}{2m + \eta (m_B-m)} +  \eta \frac{2r V_A}{m_A}\notag \\ 
& \ \ \  +  \int_{ \mathcal{T}^{\rm pure}_A \EEE} \left(2  - \frac{4a(t)}{2m + \eta (m_B-m)} - \eta \frac{2ra(t)}{m_A} \right)\, {\rm d}t
                                                                               +  \int_{ \mathcal{T}^{\rm pure}_B \EEE} \left(2  - \frac{4b(t)}{2m + \eta (m_B-m)} \right) \, {\rm d}t \\
  &\ \ \ + \int_{ \mathcal{T}^{\rm mix} \EEE} \left((2+\eta) - \frac{4(a+b)(t)}{2m + \eta (m_B-m)} -  \eta\frac{2r a(t)}{m_A}\right)  \, {\rm d}t. \label{eq: NNNNNN}
\end{align}
We  separately show that each of three integrals   in the above
right-hand side \EEE  is nonnegative. Then, the claim follows  by recalling  $V_B = rV_A$ .

 For the first one, by $r \le 1/3$ we observe that 
$$\int_{ \mathcal{T}^{\rm pure}_A \EEE} \left(2  - \frac{4a(t)}{2m + \eta (m_B-m)} - \eta \frac{2ra(t)}{m_A} \right) \, {\rm d}t \ge \int_{ \mathcal{T}^{\rm pure}_A \EEE} \left(2  - \frac{4a(t)}{2m + \eta (m_B-m)} - \eta \frac{a(t)}{m_A} \right)\, {\rm d}t, $$
and  verbatim repeat the proof of \eqref{second integra} with $a,
\mathcal{T}^{\rm pure}_A  $ in place of $b,  \mathcal{T}^{\rm pure}_B
$ and the roles of $m_A$ and $m_B$ interchanged. Note that the
reasoning relies on $m_A \le \frac{2-\eta}{2\eta}(   (2+\eta)m_B -
2m)$ whereas   \eqref{second integra}  used \EEE the analogous condition $m_B \le  \frac{2-\eta}{2\eta}(   (2+\eta)m_A - 2m  )$.

For the second  integral on the right-hand side of \eqref{eq:
  NNNNNN}, \EEE we repeat the argument in \eqref{firsti integra} and get 
\begin{align*}
\int_{ \mathcal{T}^{\rm pure}_B \EEE} \left(2  - \frac{4b(t)}{2m + \eta (m_B-m)}\right)\, {\rm d}t \ge  2\mathcal{ L \EEE}^1( \mathcal{T}^{\rm pure}_B \EEE)  \frac{(2-\eta)(m-m_B)}{2m + \eta (m_B-m)} \ge 0.
\end{align*}
For the third integral  on the right-hand side of \eqref{eq:
  NNNNNN} 
\EEE we calculate
\begin{align*}
 \int_{ \mathcal{T}^{\rm mix}} \hspace{-0.1cm} \big((2+\eta) - \frac{4(a+b)(t)}{2m + \eta (m_B-m)} -
 \eta  &\frac{2ra(t)}{m_A}\big)  \, {\rm d}t  \ge \mathcal{
                                      L\EEE}^1( \mathcal{T}^{\rm mix}) \Big(
                                      (2+\eta) - \frac{4m}{2m + \eta (m_B-m)} - 2r\eta \Big)
                                       \\ & =  \mathcal{ L\EEE}^1( \mathcal{T}^{\rm mix}) \frac{m_B (2 \eta + \eta^2 (1 - 2 r)) - m (4 \eta r + \eta^2 (1-2r ))}{2m + \eta (m_B-m)}.
\end{align*}
It suffices to show that the  numerator \EEE is nonnegative. By using  $ m_B \ge \frac{2}{2+\eta}m$, see  \eqref{eq: A2}, we estimate
\begin{align*}
m_B (2 \eta + \eta^2 (1 - 2 r)) - m (4 \eta r + \eta^2 (1-2r )) & \ge \frac{4 \eta + 2\eta^2 (1 - 2 r)}{2+\eta}m -  (4 \eta r + \eta^2 (1-2 r)) m  \\
& = \frac{ 4-\eta^2   -r(8 + 4\eta - 2\eta^2) }{2 + \eta}m \eta.
\end{align*}
Eventually, using the remaining assumptions in \eqref{eq: A2}, namely, $\eta \le \eta_{\rm triple} \sim 0.56394$ and  $r \le 1/3$, we can check
$$4-\eta^2   -r(8 + 4\eta - 2\eta^2)\ge4-\eta^2   -(8 + 4\eta - 2\eta^2)/3 \ge 4-\eta_{\rm triple}^2   -(8 + 4\eta_{\rm triple} - 2\eta_{\rm triple}^2)/3  >0. $$
This concludes the proof.


\subsection{Proof of Lemma \ref{lem:appendix}}\label{sec: lem:appendix}

 Recall the definition of $g_A^r$ and $g_B^r$ in \eqref{eq:gagb}.
We start with the proof of (a) and (d). First, assume that
\eqref{eq:assumptionsonmamb} holds  and \EEE compute the value of $g_A^r$ at $0$, namely,
\begin{equation}
g_A^r(0) = \frac{1}{r} \Big(\frac{2}{m_A} - \frac{2+\eta}{(1 + \eta\frac{p}{2}) m} \Big),
\end{equation}
which is nonpositive since $m_A \geq \frac{(2+\eta p)m}{2+\eta}$. For $\alpha \in (0,\frac{m}{m_A} - 1)$ we have $1+ \alpha \le (1+\eta\frac{p}{2}) \frac{m}{m_A}$, so
\begin{equation}
g_A^r(\alpha) = \frac{1+\alpha}{r-\alpha} \Big(\frac{2+\eta}{(1 + \alpha) m_A} - \frac{2+\eta}{(1 + \eta\frac{p}{2}) m} \Big)  \ge 0 \EEE \geq g_A^r(0).
\end{equation}
For $\alpha \in [\frac{m}{m_A} - 1, r]$ we have
\begin{equation}\label{ffflater}
g_A^r(\alpha) = \frac{1+\alpha}{r-\alpha} \Big(\frac{2+\eta}{m} - \frac{2+\eta}{(1 +\eta \frac{p}{2}) m} \Big) \ge 0 \geq g_A^r(0),
\end{equation}
 which concludes the proof of (a). The proof of (d) is identical by interchanging the roles of $A$ and $B$, and checking $g_B^r(0) \le 0$ and $g_B^r(\beta) \ge 0$ for all $\beta \in (0,1/r]$.   

 We now proceed with the proof of (b). Again, assume that \eqref{eq:assumptionsonmamb} holds. Observe that  
\begin{equation}
g_B^r(0) = \frac{2}{m_B} - \frac{2+\eta}{(1 + \eta \frac{p}{2}) m} \ge 0,
\end{equation}
since $m_B \leq \frac{(2+\eta p)m}{2+\eta}$. However,  $g_B^r$ is not necessarily minimized at $0$. In fact, for $\beta \in (0,\frac{m}{m_B} - 1]$ we have 
\begin{equation}
g_B^r(\beta) = \frac{1+\beta}{1-r\beta} \Big(\frac{2 + \eta}{(1 + \beta) m_B} - \frac{2+\eta}{(1 +  \eta\frac{p}{2}) m} \Big), 
\end{equation}
and thus
\begin{align}\label{eq mono}
(g_B^r)'(\beta) &= \frac{2+\eta}{m_B} \Big(\frac{1}{1-r\beta} \Big)' - \frac{2+\eta}{(1 + \eta\frac{p}{2}) m} \Big(\frac{1+\beta}{1-r\beta} \Big)'  =  \frac{2+\eta}{m_B} \frac{r}{(1-r\beta)^2}  - \frac{2+\eta}{(1 + \eta\frac{p}{2}) m} \frac{1+r}{(1-r\beta)^2} \notag \\
&= \frac{2+\eta}{(1-r\beta)^2 (1 + \eta\frac{p}{2}) m} \Big( r (1 + \eta\frac{p}{2}) \frac{m}{m_B} - (1+r) \Big).
\end{align}
Note that the monotonicity  of  $g^r_B$ \EEE  depends on the 
sign of \EEE  $T(\eta, p,r) :=  r (1 + \eta\frac{p}{2}) \frac{m}{m_B} - (1+r) $. On the other hand, for $\beta \in [\frac{m}{m_B} - 1, \frac{1}{r}]$ we have $m \leq (1+\beta) m_B$, so
\begin{align}\label{eq mono2}
g_B^r(\beta) = \frac{1+\beta}{1 - r\beta} \Big(\frac{2+\eta}{m} - \frac{2+\eta}{(1 + \eta\frac{p}{2}) m} \Big), 
\end{align}
which is an increasing function of $\beta$. Summarizing, \eqref{eq mono} implies that
$$ 0 \in {\rm arg \, min} \,  g_B^r    \quad \text{if  $T(\eta, p,r)  \ge 0$ }  \quad \quad  \text{and} \quad \quad   \Big\{ \frac{m}{m_B}-1\Big\} \in   {\rm arg \, min}\,  g_B^r   \quad \text{if  $T(\eta, p,r) \le  0 $.}  $$
The latter case corresponds to the case $\frac{m_B}{m} \ge  \frac{r(2+\eta p)}{2+2r}$.  Note that in the case $T(\eta, p,r) \le  0$ we have 
$$\frac{m}{m_B} - 1 \le  \frac{2+2r}{r(2+\eta p)} - 1   =   \frac{2-rp\eta}{r(2+\eta p)}  \le \frac{1}{r},$$ 
and thus $\frac{m}{m_B} - 1$ is indeed admissible. Using formula \eqref{eq mono2} we finally get that $g_B^r(\frac{m}{m_B}-1)$ is nonnegative, which finishes \EEE the proof of (b). 

 The proof of (c) runs along similar lines as the proof of (b).   Assume that \eqref{eq:assumptionsonmambprime} holds. Then, the value of $g_A^r(0)$ is nonnegative, but again it turns out that $g_A^r$ is not necessarily minimized at $0$. For $\alpha \in (0,\frac{m}{m_A} - 1]$ we have 
\begin{equation}
g_A^r(\alpha) = \frac{1+\alpha}{r-\alpha} \Big(\frac{2 + \eta}{(1 + \alpha) m_A} - \frac{2+\eta}{(1 +  \eta\frac{p}{2}) m} \Big), 
\end{equation}
so
\begin{align}\label{eqmonoprime}
(g_A^r)'(\alpha) &= \frac{2+\eta}{m_A} \Big(\frac{1}{r-\alpha} \Big)' - \frac{2+\eta}{(1 + \eta\frac{p}{2}) m} \Big(\frac{1+\alpha}{r-\alpha} \Big)'  =  \frac{2+\eta}{m_A} \frac{1}{(r-\alpha)^2}  - \frac{2+\eta}{(1 + \eta\frac{p}{2}) m} \frac{1+r}{(r-\alpha)^2} \notag \\
&= \frac{2+\eta}{(r-\alpha)^2 (1 + \eta\frac{p}{2}) m} \Big( (1 + \eta\frac{p}{2}) \frac{m}{m_A} - (1+r) \Big).
\end{align}
The monotonicity  of  $g^r_A$ \EEE depends on the  sign of
\EEE  $\tilde T(\eta, p,r) :=  (1 + \eta\frac{p}{2}) \frac{m}{m_A} - (1+r)$. For $\alpha \in [\frac{m}{m_A} - 1, r]$ we have $m \leq (1+\alpha) m_A$, so
\begin{align}\label{eq mono2XXX}
g_A^r(\alpha) = \frac{1+\alpha}{r - \alpha} \Big(\frac{2+\eta}{m} - \frac{2+\eta}{(1 + \eta\frac{p}{2}) m} \Big), 
\end{align}
which is an increasing function of $\alpha$. To sum up, we have
$$ 0 \in {\rm arg \, min} \,  g_A^r    \quad \text{if  $\tilde T(\eta, p,r)  \ge 0$ }  \quad \quad  \text{and} \quad \quad   \Big\{ \frac{m}{m_A}-1\Big\} \in   {\rm arg \, min}\,  g_A^r   \quad \text{if  $\tilde T(\eta, p,r) \le  0 $.} $$
The latter case corresponds to  $\frac{m_A}{m} \ge \frac{2+\eta
  p}{2+2r}$. Note that in the case $\tilde T(\eta, p,r) \le  0$ we
have $\frac{m}{m_A} - 1 \le r$, so $\frac{m}{m_A} - 1$ is indeed
admissible, and using formula \eqref{eq mono2XXX} we finally get that
$g_A^r(\frac{m}{m_A}-1)$ is nonnegative. This concludes the proof.

\subsection{Proof of Lemma \ref{lemma: extra}}\label{sec: lemma: extra}
\noindent \emph{Step 1: Proof of (a).}  First, we find  \EEE 
\begin{equation}
g_B^r(0) + g_A^r(0) = \frac{1}{r} \Big(\frac{2}{m_A} - \frac{2+\eta}{(1 + \eta\frac{p}{2}) m} \Big) + \frac{2}{m_B} - \frac{2+\eta}{(1 + \eta\frac{p}{2}) m}.
\end{equation}
By the definition of $p$, we have $m_B = m (1+p) - m_A$, and therefore we can rewrite as
\begin{equation}
m( g_B^r(0) + g_A^r(0)) = \frac{2}{r \frac{m_A}{m}} + \frac{2}{1 + p - \frac{m_A}{m}} - \frac{(2+\eta)(1 + \frac{1}{r})}{(1 + \eta\frac{p}{2})}.
\end{equation}
We minimize the sum of the first two addends in terms of $x=\frac{m_A}{m} \in [0,1]$. To this end, denote
\begin{equation}
h(x) = \frac{2}{rx} + \frac{2}{1 + p - x} - \frac{(2+\eta)(1 + \frac{1}{r})}{1 + \eta\frac{p}{2}}.
\end{equation}
The function $h$ is convex on $[0,1]$ with $h'(x) = -\frac{2}{rx^2} + \frac{2}{(1+p-x)^2}$. The unique solution of $h'(x) = 0$ for $x \in [0,1]$ is given by  
\begin{equation}
x_* = \frac{p+1}{\sqrt{r} + 1}
\end{equation}
 and \EEE  corresponds to the  minimum of $h$.
Therefore, it suffices to check the value of $h$ at $x_*$, and we get
\begin{equation}
m(g_A^r(0) + g_B^r(0)) \geq \frac{2(\sqrt{r}+1)}{r (p+1)} + \frac{2(\sqrt{r}+1)}{\sqrt{r}(p + 1)} - \frac{(2+\eta)(1 + \frac{1}{r})}{(1 + \eta\frac{p}{2})}.
\end{equation}
We will now  check \EEE that   the right-hand side above \EEE is positive.  By multiplying \EEE   by ${r(p+1)}/{2}$  this
corresponds to checking the inequality \EEE
\begin{equation}
\sqrt{r}+1 + r + \sqrt{r} > \frac{(2+\eta)(1+r)(1+p)}{2+\eta p},
\end{equation}
which we reorganize as
\begin{equation}
\frac{(1 + \sqrt{r})^2}{1+r} > \frac{(2+\eta)(1+p)}{2+\eta p}.
\end{equation}
The left-hand side is increasing for $r \in [0,1]$, and thus, for $r
\ge  \bar r (\eta)=\max \{r_{12}(\eta),r_{13}(\eta)\}$,  greater or equal  than $(1+\sqrt{ \bar r \EEE(\eta)})^2(1+ \bar r \EEE(\eta))^{-1} $. The right-hand side is increasing in $p$ for $\eta \in [0,2]$ and thus smaller or equal to
$\frac{3(2+\eta)}{4+\eta }$ for $p \in [0,1/2]$. We can check directly that  $(1+\sqrt{ \bar r \EEE(\eta)})^2(1+ \bar r \EEE(\eta))^{-1} - \frac{3(2+\eta)}{4+\eta }>0$ for $\eta \in (0,2)$. 


\noindent \emph{Step 2: Proof of (b).} We now come to the second case. We compute
\begin{align*}
 g^r_A(0) + g_B^r \bigg(\frac{m}{m_B} - 1 \bigg) &  = \frac{1}{r} \Big(\frac{2}{m_A} - \frac{2+\eta}{(1 + \eta\frac{p}{2}) m} \Big) + 
\frac{m}{m_B-rm + rm_B} \Big(\frac{2+\eta}{m} - \frac{2+\eta}{(1 + \eta\frac{p}{2}) m} \Big)   \\
& =  \frac{(4+2\eta p)m-(4+2\eta )m_A}{rm_Am(2+\eta p)} +   \frac{(2+\eta)\eta p}{(m_B-rm + rm_B)(2 + \eta p) }.
\end{align*}
Since  $m_B-rm + rm_B \ge 0$ by $g_B^r(\frac{m}{m_B} - 1) \ge 0$ (see \eqref{eq mono2}), to check positivity of this expression, it suffices to show
\begin{align}\label{eq: to check}
 ((4+2\eta p)m-(4+2\eta )m_A) (m_B-rm + rm_B) + (2+\eta)\eta p r  m_Am > 0.
\end{align}
 (The equality case is indeed only possible in (iii),(iv) below if $m = m_A$ and $r = r_{13}(\eta)$.) \EEE To this end, we again use the relation $m_B = (1+p)m - m_A$  and find
\begin{align*}
&((4+2\eta p)m-(4+2\eta )m_A) (m_B-rm + rm_B) + (2+\eta)\eta p r  m_Am  \notag  \\  & =   ((4+2\eta p)m-(4+2\eta )m_A) ( (1+r)(1+p)m - (1+r)m_A -rm ) + (2+\eta)\eta p r  m_Am  \\ & =: m^2 E(x;\eta,p,r)   
\end{align*} 
where $x = \frac{m_A}{m}$ and 
\begin{align}\label{eq: express}
 E(x;\eta,p,r) = a(\eta,p,r) x^2 + b(\eta,p,r) x + c(\eta, p,r)
\end{align}
with
\begin{align*}
a(\eta, p, r) &:= (4+2\eta)(1+r), \\
b(\eta,p,r)&:= -(4+2\eta)(1+r)(1+p) + r(4+2\eta) - (4+2\eta p)(1+r) + (2+\eta)\eta p r, \\
c(\eta,p,r) &:= (4+2\eta p)((1+r)(1+p)-r). 
\end{align*}
Since $\eta \in (0,2)$, one can check that  $a(\eta, p, r) > 0$, $b(\eta, p, r) < 0$, and $c(\eta, p, r) > 0$.  By assumption we have $\frac{m_B}{m} \ge  \frac{r(2+\eta p)}{2+2r}$ and thus 
$$\frac{m_A}{m}  = 1+p - \frac{m_B}{m} \le 1+p - \frac{r(2+\eta p)}{2+2r}.$$
Setting $x_{\rm max} = 1 + p - \frac{r(2+\eta p)}{2(1+r)}$, we need to check that
$$\inf_{x \in [0,x_{\rm max}]} E(x;\eta,p,r) > 0. $$
 As $E(x;\eta,p,r)$ is a second-order  polynomial \EEE \EEE in $x$,  one
has \EEE
\begin{align}
&{\rm \arg \, min} \, E(x;\eta,p,r) = \frac{-b(\eta,p,r)}{2a(\eta,p,r)}
  =:x_{\rm opt} \nonumber\\
  &\quad \text{with}   \quad E(x_{\rm opt};\eta,p,r) = \frac{4a(\eta, p, r) c(\eta, p, r) - b(\eta, p, r)^2}{4a(\eta, p, r)}. \label{eq: mittelremime}
\end{align}
Unfortunately, we cannot show directly that $E(x_{\rm opt};\eta,p,r)$ is positive for all $\eta,p,r$. Instead, we will distinguish between cases depending on the values of $x_{\rm max}$ and $x_{\rm opt}$. Observe that, since $a(\eta, p, r) > 0$ and $b(\eta, p, r) < 0$, we have that $x_{\rm opt} > 0$. One of the following situations applies: \EEE
\begin{align*}
{\rm (i)}  \  \  &  \text{$x_{\rm opt} > x_{\rm max}$ and $x_{\rm max} \leq 1$}, \\
{\rm (ii)} \ \ &  \text{$x_{\rm opt} \leq x_{\rm max}$ and $x_{\rm opt} \leq 1$}, \\ 
{\rm (iii)}  \ \ &  \text{$x_{\rm opt} > x_{\rm max}$ and $x_{\rm max} > 1$}, \\ 
{\rm   (iv)}  \ \  &  \text{$x_{\rm opt} \leq x_{\rm max}$ and $x_{\rm opt} > 1$}.
\end{align*}
Taking into account only admissible values of $x$, i.e., that $x \in [0,\min \{ x_{\rm max}, 1 \}]$, in case (i) the function $E(x;\eta,p,r)$ is minimized at $x_{\rm max}$, in case (ii) it is minimized at $x_{\rm opt}$, and in cases (iii) and (iv) it is minimized at $1$.

Let us first consider case (i). Since $x_{\rm max} < x_{\rm opt}$, we know that the minimum of $E(\cdot;\eta,p,r)$ on $[0,x_{\rm max}]$ is attained at $x_{\rm max}$. An explicit calculation yields
$$E\Big( 1 + p - \frac{r(2+\eta p)}{2(1+r)};\eta,p,r\Big) = \eta p r (\eta p +2) \ge 0.$$
 In particular, \EEE in case (i)  the term $E(x;\eta,p,r)$
\EEE is nonnegative, and it is equal to zero if and only if $p = 0$. We show that $\inf_{x \in [0,x_{\rm max}]} E(x;\eta,p,r) = 0 $ yields to a contradiction. In this case, the above arguments show that necessarily $p=0$, $\frac{m_A}{m} = \frac{2}{2(1+r)}$, and $\frac{m_B}{m} = \frac{2r}{2(1+r)}$.  As by assumption we also have $\frac{m_B}{m} \le \frac{\eta}{2+\eta}$, we deduce $r \le \frac{\eta}{2}$. This however contradicts the assumption  $r > r_{12}(\eta) = \frac{\eta}{2}$.

In case (ii), we need to check that $4 a(\eta,p,r) c(\eta,p,r) -b(\eta,p,r)^2 > 0$, see \eqref{eq: mittelremime}. One computes
$$ F (p;\eta,r) := 4 a(\eta,p,r) c(\eta,p,r) -b(\eta,p,r)^2  = \tilde{a}(\eta,r)p^2 + \tilde{b}(\eta,r) p+ \tilde{c}(\eta,r),$$
where
\begin{align*}
\tilde{a}(\eta,r) & = -\eta^4 r^2 + 4\eta^3 r^2 + 8\eta^3 r + 20\eta^2 r^2 + 24\eta^2 r + 16\eta r^2 + 16\eta r - 16r^2 - 32r - 16, \\
\tilde{b}(\eta,r) & = 4\eta^3 r + 8\eta^2 r^2 + 24\eta^2 r + 16\eta r^2 + 16\eta r - 16\eta + 32r^2 + 32r, \\ 
\tilde{c}(\eta,r)  & = - 4 (\eta - 2r)^2.
\end{align*}
 This is  a second-order  polynomial \EEE in \EEE $p$. Since $\tilde{c}(\eta,r) < 0$, it suffices to check that there exist $0 < p_1 < p_2$ such that $F(p_1;\eta,r) > 0$ and $F(p_2;\eta,r) >0 $. Then, we automatically get $F(p;\eta,r) > 0$ for all $p \in [p_1,p_2]$.

We choose the values $p_1$ and $p_2$ such that $[p_1,p_2]$ covers all admissible values for $p$: we set  \EEE
$$p_1 :=  \EEE \frac{4r-2\eta}{4+4r-\eta^2r-2\eta r} \quad \mbox{and} \quad  p_2:=\frac{4\eta r + 2\eta +4r}{2\eta r + 4\eta + 4r + 4 - \eta^2 r}.$$
Then, a computation yields that  $x_{\rm opt} - x_{\rm max}$ equals 
\begin{align*}
&\frac{(4+2\eta)(1+r)(1+p) - r(4+2\eta) + (4+2\eta p)(1+r) - (2+\eta)\eta p r}{2(4+2\eta)(1+r)} - \Big( 1 + p - \frac{r(2+\eta p)}{2(1+r)}\Big) \\
& = \frac{p(\eta^2  r + 2 \eta  r - 4  r  - 4)           - 2 \eta + 4 r}{4(\eta+ 2) (r+1)  },
\end{align*}
i.e.,  $x_{\rm max}  \ge \EEE x_{\rm opt}$ if and only if $p  \ge  \EEE p_1 =  \frac{4r-2\eta}{4+4r-\eta^2r-2\eta r}$. (Here, use that $4+4r-\eta^2r-2\eta r >0$.) Secondly, we have
$$x_{\rm opt} \le 1 \ \Leftrightarrow \  2a(\eta,p,r) + b(\eta,p,r) \ge 0  \ \Leftrightarrow \   p (-4 -4\eta -4r - 2 \eta  r + \eta^2 r) +  ( 2 \eta  + 4 r + 4 \eta r)    \ge 0,   $$
i.e.,
$$x_{\rm opt} \le 1 \ \Leftrightarrow \   p \le p_2 = \frac{4\eta r + 2\eta +4r}{2\eta r + 4\eta + 4r + 4 - \eta^2 r}.  $$
Therefore, the interval $[p_1,p_2]$ covers all $p$ such that $x_{\rm opt} \leq x_{\rm max}$ and $x_{\rm opt} \leq 1$, i.e., the whole case (ii).  We now compute
$$ F(p_1;\eta,r) =  \frac{32 (-2 + \eta) \eta (2 + \eta)^2 (\eta - 2 r) r (1 + r)^2}{(-4 + (-4 + \eta (2 + \eta)) r)^2},$$
which is indeed positive whenever $r > r_{12}(\eta) =   {\eta}/{2}$
for $\eta \in (0,2)$.  On the other hand, \EEE
$$ F(p_2;\eta,r) \EEE =  -\frac{16 \eta (\eta + 2)^2 (r + 1)^2 (   ((\eta-6)\eta (\eta +2) - 8)r^2 - 2\eta(3\eta +2) r + 4\eta     )}{(\eta^2 r - 2 r (\eta + 2) - 4 (\eta + 1))^2}  $$
 and we are left with checking \EEE the positivity of the expression
$$-(   ((\eta-6)\eta (\eta +2) - 8)r^2 - 2\eta(3\eta +2) r + 4\eta     ). $$
This is   positive   if $r > \frac{2 \eta + 3 \eta^2 - \sqrt{\eta (2 +
    \eta)^2 (8 + 5 \eta)}}{-8 - 12 \eta - 4 \eta^2 + \eta^3}$, 
which holds \EEE true since one can indeed calculate that 
$${r \ge  \bar r \EEE(\eta) = \max\Big\{ \frac{\eta}{2},  \frac{8 +
  2\eta - 4 \sqrt{4 + 2\eta}}{\eta^2 - 8 - 2\eta + 4 \sqrt{4 + 2\eta}}
\Big\}  \geq \EEE  \frac{2 \eta + 3 \eta^2 - \sqrt{\eta (2 + \eta)^2 (8 + 5 \eta)}}{-8 - 12 \eta - 4 \eta^2 + \eta^3}} $$ 
for all  $\eta \in (0,2)$, by distinguishing the cases $\eta \ge \eta_{\rm triple}$ and $\eta \le \eta_{\rm triple}$.
Finally, we address cases (iii) and (iv), in which $E(\cdot;\eta,p,r)$ is minimized for $x = 1$.  We compute
$$E(1;\eta,p,r)/\eta =  p^2(2   r  + 2    ) + p ( \eta r  - 2   r  - 2 ) + 2r. $$
Minimizing with respect to $p$ gives the estimate
$$E(1;\eta,p,r)/\eta \ge \frac{1}{8(1+r)} \big( 8(2   r  + 2    ) r  -  ( \eta r  - 2   r  - 2 )^2\big).  $$
We compute that 
$$  8(2   r  + 2    ) r  -  ( \eta r  - 2   r  - 2 )^2 =   - r^2 (-12 - 4 \eta + \eta^2) + 4 r (2 + \eta) - 4$$
is nonnegative  if $r \ge r_*(\eta) :=  \frac{2 (-2 - \eta + 2
  \sqrt{2} \sqrt{2 + \eta})}{12  +  4 \eta - \eta^2}$ and positive if
$r > r_*(\eta)$.  By using \EEE that   $r_*(\eta) = r_{13}(\eta)$ (follows after some computation) \EEE
this shows $E(x;\eta,p,r)>0$ if $x = \frac{m_A}{m} <1$ or $r >
r_{13}(\eta)$  and \EEE  concludes the proof of point (b).  

\noindent \emph{Step 3: Proof of (c).} We prove this point with a
technique  similar \EEE to the previous one, but with a slightly different structure of the proof. First, we compute
\begin{align*}
g^r_A \bigg(\frac{m}{m_A} - 1 \bigg) + g_B^r(0) &= \frac{m}{rm_A- m + m_A} \Big(\frac{2+\eta }{m} - \frac{2+\eta}{(1 + \eta\frac{p}{2}) m} \Big) + \Big(\frac{2}{m_B} - \frac{2+\eta}{(1 + \eta\frac{p}{2}) m} \Big)   \\
&= \frac{(2+\eta)\eta p}{(rm_A- m + m_A)(2 + \eta p) } + \frac{(4+2\eta p)m-(4+2\eta )m_B}{m_B m(2+\eta p)}.
\end{align*}
Since $r m_A - m + m_A \ge 0$ by $g_A^r(\frac{m}{m_A} - 1) \ge 0$ (see \eqref{ffflater}), to check positivity of this expression, it suffices to show
\begin{align}
((4+2\eta p)m-(4+2\eta )m_B) (r m_A - m + m_A) + (2+\eta)\eta p  m_B m > 0.
\end{align}
To this end, we again use the relation $m_A = (1+p)m - m_B$ and find
\begin{align*}
&((4+2\eta p)m-(4+2\eta )m_B) (r m_A - m + m_A) + (2+\eta)\eta p m_B m  \notag  \\  & =   ((4+2\eta p)m-(4+2\eta )m_B) ( (1+r)(1+p)m - (1+r)m_B - m) + (2+\eta)\eta p m_B m  \\ & =: m^2  \hat E \EEE(x;\eta,p,r),  
\end{align*} 
where $x = \frac{m_B}{m}$ and 
\begin{align}
 \hat E \EEE(x;\eta,p,r) =  \hat a \EEE(\eta,p,r) x^2 +  \hat b \EEE(\eta,p,r) x +  \hat c \EEE(\eta, p,r)
\end{align}
with
\begin{align*}
 \hat a \EEE(\eta, p, r) &:= (4+2\eta)(1+r), \\
 \hat b \EEE(\eta,p,r) &:= -(4+2\eta)(1+r)(1+p) + (4+2\eta) - (4+2\eta p)(1+r) + (2+\eta)\eta p, \\
 \hat c \EEE(\eta,p,r) &:= (4+2\eta p)((1+r)(1+p)-1). 
\end{align*}
Again, we  observe that $ \hat a \EEE(\eta,p,r) = a(\eta, p, r) > 0$, $ \hat b \EEE(\eta,p,r) < 0$, and $ \hat c \EEE(\eta,p,r) = c(\eta,p,r)   > 0$. By assumption we have $\frac{m_A}{m} \ge  \frac{2+\eta p}{2+2r}$, and thus 
$$\frac{m_B}{m} = 1+p - \frac{m_A}{m} \le 1+p - \frac{2+\eta p}{2+2r}.$$
 Setting $ \hat x_{\rm max} \EEE = 1 + p - \frac{2+\eta p}{2(1+r)}$, we need to check that
\begin{align}\label{XEEE}
\inf_{x \in [0, \hat x_{\rm max} \EEE]}  \hat E \EEE(x;\eta,p,r) > 0. 
\end{align}
We observe that $ \hat E \EEE(x;\eta,p,r)$ is  a second-order
 polynomial \EEE in \EEE $x$ and thus
\begin{align}
{\rm \arg \, min} \,  \hat E \EEE(x;\eta,p,r) = \frac{- \hat b \EEE(\eta,p,r)}{2 \hat a \EEE(\eta,p,r)}  =: \hat x_{\rm opt}  \EEE\   \text{with}   \   \hat E \EEE( \hat x_{\rm opt} \EEE;\eta,p,r) = \frac{4 \hat a \EEE(\eta, p, r)  \hat c \EEE(\eta, p, r) -  \hat b \EEE(\eta, p, r)^2}{4 \hat a \EEE(\eta, p, r)}.
\end{align}
Again, we cannot show directly that $ \hat E \EEE( \hat x_{\rm
  opt}\EEE;\eta,p,r)$ is positive for all $\eta,p,r$, and we will
instead consider two cases depending on the values of $ \hat
x_{\rm max} \EEE$ and $ \hat x_{\rm opt} \EEE$. Since \EEE $ \hat a \EEE(\eta, p, r) > 0$ and $ \hat b \EEE(\eta, p, r) < 0$, we have that $ \hat x_{\rm opt}  \EEE> 0$. Moreover, due to 
\begin{equation*}
 \hat x_{\rm max} \EEE < 1 \Leftrightarrow p < \frac{2 + \eta p}{2+2r} \Leftrightarrow p(2 + 2r - \eta) < 2,
\end{equation*}
and $\eta \in (0,2)$, $r \le 1$, $p \le 1/2$, we have that $ \hat x_{\rm max} \EEE < 1$. Therefore, one of the following conditions holds true:
\begin{align*}
{\rm (i)}  \  \  &  \text{$ \hat x_{\rm opt}  \EEE\ge  \hat x_{\rm max} \EEE$}, \\
{\rm (ii)} \ \ &  \text{$ \hat x_{\rm opt}  \EEE<  \hat x_{\rm max} \EEE < 1$}.
\end{align*}
Taking into account only admissible values of $x$, i.e.,   $x \in
[0,\min \{  \hat x_{\rm max} \EEE, 1 \}]$, in case (i) the
function $ \hat E \EEE(x;\eta,p,r)$ is minimized at $ \hat
x_{\rm max} \EEE$ and in case (ii) it is minimized at $ \hat
x_{\rm opt}$. \EEE

Let us first consider case (i). Since $ \hat x_{\rm max} \EEE \le  \hat
x_{\rm opt}$, \EEE we know that the minimum of $ \hat E \EEE(\cdot;\eta,p,r)$ on $[0, \hat x_{\rm max} \EEE]$ is attained at $ \hat x_{\rm max} \EEE$. An explicit calculation yields
$$ \hat E \Big( 1 + p - \frac{2+\eta p}{2(1+r)};\eta,p,r\Big) = \eta p (\eta p +2) \ge 0,$$
so in case (i) the energy is nonnegative, and it is equal to zero if and only if $p = 0$. We show that $\inf_{x \in [0, \hat x_{\rm max} \EEE]}  \hat E \EEE(x;\eta,p,r) = 0 $ yields to a contradiction. In this case, the above arguments show that necessarily $p=0$, $\frac{m_B}{m} = \frac{2r}{2(1+r)}$, and $\frac{m_A}{m} = \frac{2}{2(1+r)}$. As $r \le 1$, this contradicts the assumption $m_B > m_A$. \EEE

%
%

Let us check for which values of $r$ case (ii) applies. First, we compute that   $ \hat x_{\rm opt}  \EEE-  \hat x_{\rm max} \EEE$  equals
\begin{align*}
& \frac{(4+2\eta)(1+r)(1+p) - (4+2\eta) + (4+2\eta p)(1+r) - (2+\eta)\eta p}{2(4+2\eta)(1+r)} - \Big( 1 + p - \frac{2+\eta p}{2(1+r)}\Big) \\
& = \frac{p(\eta^2 + 2\eta - 4r - 4) - 2 \eta r + 4}{4(2 + \eta) (1+r)  }.
\end{align*}
Then, since $4 - 2 \eta r > 0$, we get that $ \hat x_{\rm max} \EEE  \le  \hat x_{\rm opt} \EEE$ if and only if $\eta^2 + 2\eta - 4r - 4 \geq 0$ or $\eta^2 + 2\eta - 4r - 4 < 0$ and $p \le \frac{4 - 2\eta r}{4 + 4r - \eta^2 - 2\eta}$. We let
$$p'_1 = \frac{4 - 2\eta r}{4 + 4r - \eta^2 - 2\eta}.$$
Thus, whenever $\eta^2 + 2\eta - 4r - 4 \geq 0$, we are in case (i) and the minimal energy is positive.

Therefore, from now on we  assume that $\eta^2 + 2\eta - 4r - 4 < 0$. We now distinguish between two further cases. Let us check for which $r$ (as a function of $\eta$) we have that $p'_1 \geq \frac{1}{2}$. For this range of values for $r$ we actually have $ \hat x_{\rm max} \EEE \leq  \hat x_{\rm opt} \EEE$ for all  $p \in [0,\frac{1}{2}]$, and thus again case (i) applies. We calculate that $p'_1 \geq \frac{1}{2}$
is equivalent to
\begin{equation*}
8 - 4\eta r \geq 4 + 4r - \eta^2 - 2\eta,
\end{equation*}
which in turn leads to
\begin{equation*}
r \leq \frac{4 + 2 \eta + \eta^2}{4+ 4\eta} < 1.    
\end{equation*}
Hence, for all $r \leq \frac{4 + 2 \eta + \eta^2}{4+ 4\eta}$ and all $p \le \frac{1}{2}$ we have   $ \hat x_{\rm max} \EEE  \le  \hat x_{\rm opt} \EEE$ and thus case (i) applies. 

Finally, we need to take care of the case $r > \frac{4 + 2 \eta + \eta^2}{4+ 4\eta}$ and $ \hat x_{\rm max}   \ge  \hat x_{\rm opt}$. \EEE In this case,  the function  $ \hat E \EEE(x;\eta,p,r)$ is minimized at  $ \hat x_{\rm opt} \EEE$, and thus to confirm \eqref{XEEE} we need to check that 
$$4  \hat a \EEE(\eta,p,r)  \hat c \EEE(\eta,p,r) - \hat b \EEE(\eta,p,r)^2 > 0.$$
One computes
$$  \bar F \EEE (p;\eta,r) := 4  \hat a \EEE(\eta,p,r)  \hat c \EEE(\eta,p,r) -  \hat b \EEE(\eta,p,r)^2  =  \bar a \EEE(\eta,r) p^2 +  \bar b \EEE (\eta,r) p +  \bar c \EEE(\eta,r),$$
where
\begin{align*}
 \bar a \EEE(\eta,r) & = -\eta^4 + 4\eta^3 + 8\eta^3 r + 20\eta^2 + 24\eta^2 r + 16\eta + 16\eta r - 16r^2 - 32r - 16, \\
 \bar b \EEE (\eta,r) & = 4\eta^3 r + 8\eta^2 + 24\eta^2 r + 16\eta + 16\eta r - 16\eta r^2 + 32 + 32r, \\ 
 \bar c \EEE(\eta,r)  & = - 4 (\eta r - 2)^2.
\end{align*}
This is a  second-order  polynomial \EEE \EEE in variable $p$. Since $ \bar c \EEE (\eta,r) < 0$, it suffices to check that there exist $p'_1, p'_2 > 0$ such that $ \bar F \EEE(p'_1;\eta,r) > 0$ and $ \bar F \EEE(p'_2;\eta,r) > 0$. Then, we automatically get $ \bar F \EEE(p;\eta,r) > 0$ for all $p \in [p'_1,p'_2]$. We take $p'_1$ as above, i.e., $p'_1 = \frac{4 - 2\eta r}{4 + 4r - \eta^2 - 2\eta}$ and $p'_2 = \frac12$, i.e., the maximum value admissible for $p$. An explicit computation shows that
\begin{equation*}
 \bar F \EEE(p'_1; \eta,r) = \frac{32 (\eta - 2) \eta (\eta + 2)^2 (r + 1)^2 (\eta r - 2)}{(\eta^2 + 2\eta - 4r - 4)^2} > 0,
\end{equation*}
where we use $r \le 1$ and $\eta \in (0,2)$. Another explicit calculation yields 
\begin{align*}
 \bar F \EEE(p'_2; \eta,r) &= 4 (2 + \eta) (4 + \eta) (1 + r) (1 + 3 r) - \frac14 (\eta^2 - 2 \eta (1 + 4 r) - 4 (3 + 5 r))^2.
\end{align*}
We compute the derivative   with respect to $r$ and get 
\begin{align*}
\frac{\partial}{\partial r}  \bar F \EEE(p'_2; \eta,r) &=  8 + 4 \eta^3 + \eta (28 - 16 r) + \eta^2 (18 - 8 r) - 8 r.
\end{align*}
One can verify that $\frac{\partial}{\partial r}  \bar F \EEE(p'_2; \eta,r) >0$   for all $\eta \in (0,2)$ and $r \in [0,1]$. To see this, we observe that the derivative is linear in $r$, and  for all $\eta \in (0,2)$  it satisfies 
\begin{equation*}
\frac{\partial}{\partial r}  \bar F \EEE(p'_2; \eta,r)|_{r=0} = 8 + 4 \eta^3 + 28 \eta + 18 \eta^2 > 0
\end{equation*}
and
\begin{equation*}
\frac{\partial}{\partial r}  \bar F \EEE(p'_2; \eta,r)|_{r=1} = 2 \eta (6 + 5 \eta + 2 \eta^2) > 0.
\end{equation*}
Therefore, $ \bar F \EEE(p'_2; \eta,r)$ is increasing as a function of $r$ for every fixed $\eta \in (0,2)$. As we treat the case $r > \frac{4 + 2 \eta + \eta^2}{4+ 4\eta}$, it is thus enough to compute the value of $ \bar F \EEE(p'_2; \eta, r)$ at $r = \frac{4 + 2 \eta + \eta^2}{4+ 4\eta}$. We find
\begin{align*}
 \bar F \EEE\left(p'_2; \eta, \frac{4 + 2 \eta + \eta^2}{4+ 4\eta}\right) &= \frac{\eta (8 + 6 \eta + \eta^2)^2}{2 (1 + \eta)} >0
\end{align*}
for all $\eta \in (0,2)$.
Hence, we have that $ \bar F \EEE(p'_2; \eta, r)$ is positive for
all $r > \frac{4 + 2 \eta + \eta^2}{4+ 4\eta}$, so $ \bar F
\EEE(p;\eta,r) > 0$ for all $p \in [p'_1,p'_2]$ for all admissible
$\eta$ and $r$. This in turn implies that $ \hat E \EEE( \hat
x_{\rm opt} \EEE;\eta,p,r)>0$, which concludes the proof.

\subsection{ Proof of Lemma \ref{lemma: letztes lemma}}\label{sec: lemma: letztes lemma}
We argue by contradiction. First, by  $m_A = m$, we get $m_B = pm$, and thus $p= \frac{m_B}{m} \ge  \frac{r(2+\eta p)}{2+2r}$ yields $p \ge \frac{2r}{2+(2-\eta)r}$. Then, we get 
$$p + p' > \frac{2r}{2+(2-\eta)r}   +  \frac{2-\eta}{2} = \frac{-2 - (-2 + \eta) \eta + \sqrt{2} \sqrt{2 + \eta}}{2 \eta} ,$$
where in the last step we used $r = r_{13}(\eta)$. An elementary
computation shows that this value is larger than
$2\frac{\sqrt{4+2\eta} -2 }{\eta}$ for all $\eta \in (0, \eta_{\rm
  triple})$. This contradicts Lemma \ref{lem:psmaller13}.

\section*{Acknowledgements}
MF  acknowledges support of the DFG project FR 4083/3-1.  This work
 was supported by the Deutsche Forschungsgemeinschaft (DFG, German
 Research Foundation) under Germany's Excellence Strategy EXC 2044
 -390685587, Mathematics M\"unster: Dynamics--Geometry--Structure.  WG acknowledges support of the Austrian Science Fund (FWF), grants I4354 and ESP 88.  US acknowledges support of the FWF grants I4354, F65, I5149, and P\,32788. For the purpose of open access, the authors have applied a CC BY public copyright licence to any Author Accepted Manuscript version arising from this submission.
\EEE

\end{document}